\documentclass[11pt, reqno]{amsart}

\usepackage{amsmath,amssymb,latexsym}
\usepackage{color}
\usepackage{xcolor}
\usepackage{graphicx}
\usepackage{cite}
\usepackage[colorlinks=true]{hyperref}
\hypersetup{urlcolor=blue, citecolor=red, linkcolor=blue}
\usepackage{microtype}

\newtheorem{theorem}{Theorem}[section]
\newtheorem{lemma}{Lemma}

\newtheorem{corollary}{Corollary}
\newtheorem*{main-theorem}{Main Theorem}
\newtheorem*{remark*}{Remark}
\newtheorem*{lemma*}{Lemma A.1}

\numberwithin{equation}{section}
\usepackage{verbatim}

\begin{document}

	\title[1D nonlinear shallow water equations]{Refined wave breaking for the one-dimensional nonlinear shallow water equations}
	
	\author{Pingchun Liu, Jean-Claude Saut, Shihan Sun, and Yuexun Wang}

	\address{School of Mathematics and Statistics, Lanzhou University, 730000 Lanzhou,
		China}
	
	\email{pingchunliu2022@lzu.edu.cn}

	\address{Université Paris-Saclay, CNRS, Laboratoire de Mathématiques d'Orsay, 91405 Orsay, France}
	\email{jean-claude.saut@universite-paris-saclay.fr}

	\address{School of Mathematics and Statistics, Lanzhou University, 730000 Lanzhou,
		China}
	
	\email{sunshh2024@lzu.edu.cn}

	\address{School of Mathematics, 
		Capital Normal University, 
		100048 Beijing, China}
	
	\email{yuexunwang@cnu.edu.cn}
	


	\begin{abstract}
		This paper aims to give a refined wave breaking description of the Cauchy problem to the one-dimensional nonlinear shallow water equations  providing a sharp  estimate of the lifespan of the solutions  depending on the amplitude and topography parameters, under a  non-cavitation condition which excludes the scenario that the solutions have compact support. 
		We construct smooth initial data with finite $\dot{H}^5$-norm such that the $L^\infty$-norm of the spatial derivative of the solution blows up at one single point in finite time with a precise blowup profile. 
     
	\end{abstract}

\maketitle

	\section{Introduction}\label{sec1}
	The one-dimensional nonlinear shallow water equations (abbreviated as `NSW') equations are given by 
	\begin{equation}\label{eq:1.1}
	\left\{
	\begin{aligned}
	&\partial_{t} \zeta+\partial_{x}\big((H+\varepsilon\zeta-\beta^{*} b) \bar{v}\big)=0, \\
	&\partial_{t} \bar{v}+\varepsilon \bar{v} \partial_{x} \bar{v}+\partial_{x} \zeta=0,
	\end{aligned}
	\right.
	\end{equation}
	where $(x,t)\in\mathbb{R}\times \mathbb{R}^+$, and $\zeta$ and $\bar{v}$ are the unknown functions of $(x,t)$ representing the elevation of the surface and the vertically averaged horizontal velocity of the water, respectively. Here, $b(x)$ is the bottom topography, $H>0$ is the depth at rest of the water, 
	$\varepsilon$ and $\beta^{*}$ denote the amplitude and topography parameters, respectively, satisfying 
    \begin{equation*}
        0 \leq \varepsilon, \beta^{*} \leq 1.
    \end{equation*}
    The NSW equations are also called the Saint-Venant system. They can be derived from the full water waves system. Actually, they are the first order approximation of this system when neglecting the  terms of order $\mathcal{O}(\mu)$ where $\mu$ is the shallowness parameter. 
    
	We refer the readers to \cite{Duchene,MR3060183,MR4105349} for the derivation and full justification of the NSW equations as approximations of the water waves system. To study the Cauchy problem to \eqref{eq:1.1}, we impose the initial data
	\begin{equation}\label{initial data}
	(\zeta, \bar{v})|_{t=0}=(\zeta_0, \bar{v}_0).
	\end{equation}
	
	Let $h=H+\varepsilon\zeta-\beta^{*}  b$, which denotes the total depth of the water and thus is positive. (Notice that $\zeta$ does not have to be positive, and can even be negative at some point). We assume that  the initial elevation satisfies the following non-cavitation condition:
	\begin{equation}\label{eq:non-cavitation}
	h|_{t=0}=H+\varepsilon\zeta_0-\beta^{*} b>h_{\text{min}}
	\end{equation}
	for some constant $h_{\text{min}}>0$. 
	Following the classical theory of symmetrizable hyperbolic systems \cite{MR1172111,MR3060183}, the Cauchy problem \eqref{eq:1.1}-\eqref{initial data} under \eqref{eq:non-cavitation} is locally well-posed for times of order $\mathcal{O}(1/\max\{\varepsilon,\beta^* \})$ if the initial data $(\zeta_0, \bar{v}_0)$ belong to $H^s(\mathbb{R})$ with $s>3/2$. By the Lax's argument \cite{MR165243,Duchene}, \eqref{eq:1.1}-\eqref{initial data} under \eqref{eq:non-cavitation} can exhibit wave breaking (that is, the gradient of the solutions blows up while the solutions themselves stay bounded, which is also called shock formation or gradient blowup) at times of order $\mathcal{O}(1/\max\{\varepsilon,\beta^* \})$. Together, these two facts tell that the lifespan of the classical solutions to \eqref{eq:1.1}-\eqref{initial data} has sharp times of order $\mathcal{O}(1/\max\{\varepsilon,\beta^* \})$.

	This paper studies the Cauchy problem for the one-dimensional nonlinear shallow water equations \eqref{eq:1.1}–\eqref{initial data}. Under the non-cavitation condition \eqref{eq:non-cavitation}, we establish a refined description of wave breaking and prove that the lifespan of classical solutions is sharply of order $\mathcal{O}(1/\max\{\varepsilon,\beta^*\})$. We construct smooth initial data with finite $\dot{H}^5$-norm such that the $L^\infty$-norm of the spatial derivative of the solution blows up at a single point in finite time with a precise blow-up profile by borrowing ideas of the blowup technique based on the modulation theory in self
	similar variables used in \cite{MR4465909,MR4612575, MR4612576,MR4493613,MR4468858}, which were generalized to different contexts of PDEs \cite{Bae2024,MR4972958,MR4664477,Kim2024,MR4913083,MR4855742,MR4752990,Qiu2021,MR4752332,MR4851892,MR4321245}. The singularity occurs at a single spatial point, whose location and blowup time can be explicitly computed. The blowup profile is a cusp with $C^{1/3}$ regularity, and the singularity is described by an asymptotically self-similar shock profile that is stable under the $\dot{H}^5$ topology.  
  Central to our approach is a detailed characterization of the far-field behavior of the self-similar ansatz $(W,Z)$, which enables us to overcome the difficulties  arising from the non-cavitation condition, which prevents the Riemann invariants $(w, z)$ from having a compact support. The sharp lifespan scaling arises because the transformations cause the $\dot H^5$-norm of the initial data to be of order $\max\{\varepsilon,\beta^* \}$, as verified by a Moser-type estimate, which yields the stated time scale  by the local well-posedness theory. Furthermore, our analysis provides the first treatment of the non-flat bottom case ($b\neq 0$). This case requires the bottom topography to have a  higher regularity than the solution $(w,z)$  by \eqref{eq:2.2}, specifically, $b(x)\in \dot{H}^6(\mathbb{R})$. This requirement is critical for establishing the estimates on $\partial_{y} W$ in the middle field ($\ell \leq |y| \leq \frac{1}{2} e^{\frac{3}{2}s}$) and the far field ($\frac{1}{2} e^{\frac{3}{2}s}\leq |y| < \infty$): the growth of the weights there outpaces the decay of the bottom topography term. The necessary additional decay is obtained by applying the chain rule $\frac{d}{dy}= e^{-3s/2}\frac{d}{dx}$, which converts the high spatial regularity of $b$ into decay in $y$. We refer the readers to \cite{Chen2025,MR3799059,MR4913083} for more related blowup results on the shallow water-type equations.

	Notably, the NSW equations neglect the dispersion effects of the full water waves system. 
	By keeping terms up to order $\mathcal{O}(\mu)$, one will get the Serre–Green–Naghdi (abbreviated as `SGN') equations, to which the local well-posedness was extensively studied \cite{MR2400253,MR3851335,MR2237287}. The SGN equations contain third order dispersive terms that may play a regularizing role, which makes the question of global well-posedness therefore becomes relevant, see the numerical computations \cite{MR4470529}. The recent works \cite{MR4459027,MR4740626} considered the SGN equations with surface tension, showing the global existence of dissipative weak solutions and presenting a finite blowup result.  
	There are also many other water wave models taking dispersion effects into consideration, such as the Whitham-Boussinesq system \cite{MR2991247,MR3763731,MR3060183,MR3902471,MR3390078,MR3668593,MR4248633}  that coincides with the abcd system at low frequencies introduced in \cite{MR1915939,MR2057134}, for which the blowup of the solutions remains open \cite{KS}.

	The paper is organized as follows. In Section \ref{sec2}, we reformulate \eqref{eq:2.1} using Riemann invariants and self-similar transformations, and provide the evolution of higher-order derivatives and constraints on $W$. The assumptions on the initial data and the statement of the main results are given in Section \ref{sec3}. Section \ref{sec4} establishes the bootstrap assumptions and derives several consequences. 
	In Section \ref{sec7}, we define the Lagrangian trajectories of $W$ and $Z$ and establish control over them. Sections \ref{sec6}, \ref{sec8} and \ref{sec9} are devoted to closing the bootstrap assumptions.
	The proof of the main theorem (Theorem \ref{thm1}) is given in Section \ref{sec10}. 

	\section{The Reformulation of the Main Problem}\label{sec2}
	
	\subsection{Riemann invariants and self-similar transforms}\label{sec2.1}
	Let
	\begin{equation}\label{eq:transform}
	h=H+\varepsilon\zeta-\beta^{*} b\quad \text{and}\quad  u=\varepsilon \bar{v}.
	\end{equation}
	Notice that $\zeta$ has no fixed sign, but $h$ is always positive up to times of order $\mathcal{O}(1/\max\{\varepsilon,\beta^* \})$
	by the local-in-time well-posedness theory of \eqref{eq:1.1}-\eqref{initial data} under the non-cavitation condition \eqref{eq:non-cavitation}.  
	Thus, we rewrite \eqref{eq:1.1} by using the variables $(h, u)$  in the following form:
	\begin{equation}\label{eq:height}
	\begin{cases}
	\partial_{t} u+u \partial_{x}u+\partial_{x} h+\beta^{*} \partial_{x}b=0, \\
	\partial_{t}h+\partial_{x}(h u)=0.
	\end{cases}
	\end{equation}
	It is worth noting that one cannot impose compact support assumption on $h|_{t=0}$ in the study of wave breaking of \eqref{eq:height} due to \eqref{eq:non-cavitation}, although \eqref{eq:height} shares a similar structure with the one-dimensional compressible Euler equations. This makes our argument different from \cite{MR4612575, MR4612576} which relies crucially on the compact assumptions on initial data. 
	
	Setting 
    \begin{equation}\label{eq:transform2}
        \sigma=2h^{\frac{1}{2}},\quad u=u,
    \end{equation}
    we can symmetrize the system \eqref{eq:1.1} as follows:
	\begin{equation}\label{eq:2.1}
	\begin{cases}
	\partial_{t}u+u\partial_x u+\frac{1}{2}\sigma\partial_{x}\sigma=-\beta^{*}  \partial_{x}b,\\
	\partial_{t}{\sigma}+u\partial_x\sigma+\frac{1}{2}\sigma\partial_{x}u=0.
	\end{cases}
	\end{equation}
	Then, we define the (scaled) Riemann invariants $(w, z)$ by 
	\begin{equation}\label{eq:transform3}
	w=\frac{3}{4}(u+\sigma),\ \ z=\frac{3}{4}(u-\sigma).
	\end{equation}
	This allows us to further diagonalize the system \eqref{eq:2.1}, resulting in:
	\begin{equation}\label{eq:2.2}
	\begin{cases}
	\partial_t w+( w+\frac{1}{3} z)\partial_{x} w=-\frac{3}{4}\beta^{*}  \partial_{x}b, \\
	\partial_t z+(\frac{1}{3}w+z)\partial_{x} z=-\frac{3}{4}\beta^{*}  \partial_{x}b.
	\end{cases}
	\end{equation}
	Since \eqref{eq:2.2} is translation invariant in time, one can choose $t_{0}=-\delta$ as the initial time. 
	By the non-cavitation condition \eqref{eq:non-cavitation}, it is straightforward to check that if $(\zeta_0, \bar{v}_0) \in H^5(\mathbb{R})$, then $(w_0, z_0) \in \dot{H}^5(\mathbb{R})$. Therefore, the Cauchy problem to \eqref{eq:2.2} with the initial data $(w_0, z_0) \in \dot{H}^5(\mathbb{R})$ and $b(x)\in \dot{H}^6(\mathbb{R})$ is locally well-posed in $C([-\delta, T_*), \dot{H}^5(\mathbb{R})) \cap C^1([-\delta, T_*), \dot{H}^4(\mathbb{R}))$. 
	We will assume that $T_{*} $ is the maximal time of existence in this solution class in the rest of this paper for convenience.
	
	First, we introduce three dynamic modulation variables
	\begin{equation*}
	\xi, \tau, \kappa:[-\delta, T_{*}] \rightarrow \mathbb{R},
	\end{equation*}
	which serve as parameters to track the location, time, and amplitude of the wave breaking, satisfying
	\begin{equation}\label{eq:2.3}
	\xi(-\delta)=0,\quad \tau(-\delta)=0,\quad \kappa(-\delta)=\kappa_{0},
	\end{equation}
	and
	\begin{equation}\label{eq:2.4}
	\xi(T_{*})=x_{*}, \quad \tau(T_{*})=T_{*},
	\end{equation}
	where $T_{*} $ and $x_{*}$  are the blow-up time and location, respectively. We will later show that  $T_{*}$  is the unique fixed point of $\tau $. 
	Next, we introduce the following self-similar transformation
	\begin{equation}\label{eq:2.5}
	\begin{aligned}
	&s=s(t)=-\log(\tau(t)-t),\\
	&y=y(x,t)=\frac{x-\xi(t)}{(\tau(t)-t)^{\frac32}}
	\end{aligned}
	\end{equation}
	and self-similar ansatz
	\begin{equation}\label{eq:2.6}
	\begin{aligned}
	w(x,t)&=e^{-\frac s2}W(y,s)+\kappa(t),\\
	z(x,t)&=Z(y,s).
	\end{aligned}
	\end{equation}
    We also change the function $\partial_{x}b$ to self-similar coordinates by letting $$\partial_{x}b(x)=B(y,s).$$
	Finally, we will derive the generating equations of the unknowns $(W, Z)$. To ease notation, let 
	\begin{equation*}
	\beta_\tau=\beta_\tau(t)=\frac{1}{1-\dot\tau(t)}.
	\end{equation*}
	In terms of \eqref{eq:2.5} and \eqref{eq:2.6}, the system \eqref{eq:2.2} becomes
	\begin{equation}\label{eq:2.7}
	\left\{\begin{aligned}
	&\bigg(\partial_s-\frac{1}{2}\bigg)W+\mathcal{V}_W\cdot\partial_yW =-\frac{3}{4}\beta^{*}  \beta_\tau e^{-\frac{s}{2}}B-e^{-\frac s2}\beta_\tau\dot\kappa, \\
	&\partial_sZ+\mathcal{V}_Z\cdot\partial_yZ =-\frac{3}{4}\beta^{*}  \beta_\tau e^{-s}B.
	\end{aligned}
	\right.
	\end{equation}
	Here, the transport velocities $\mathcal{V}_W$ and $\mathcal{V}_Z$ for $W$ and $Z$, respectively, are defined as follows:
	\begin{equation}\label{eq:2.8}
	\begin{aligned}
	\mathcal{V}_W=\frac{3}{2}y+\beta_\tau W+G_W
	\end{aligned}
	\end{equation}
	with
	\begin{equation}\label{eq:2.9}
	\begin{aligned}
	G_W=e^{\frac{s}{2}}\beta_\tau\bigg(\kappa+\frac{1}{3}Z-\dot{\xi}\bigg),
	\end{aligned}
	\end{equation}
	and
	\begin{equation}\label{eq:2.10}
	\begin{aligned}
	\mathcal{V}_Z=\frac{3}{2}y+\frac{1}{3}\beta_\tau W+G_Z
	\end{aligned}
	\end{equation}
	with
	\begin{equation}\label{eq:2.11}
	\begin{aligned}
	G_Z=e^{\frac{s}{2}}\beta_\tau\bigg(\frac{1}{3}\kappa+Z-\dot{\xi}\bigg).
	\end{aligned}
	\end{equation}

	\subsection{The stable blowup self-similar Burgers profile}\label{sec2.2}
	
	The inviscid Burgers equation
	\begin{equation*}
	\partial_t u+u\partial_x u=0
	\end{equation*}
	has a family of self-similar solutions
	\begin{equation*}
	u(x,t)=(-t+T_*)^{\frac{1}{2i}}W_i\Bigg(\dfrac{x-x_*}{(-t+T_*)^{\frac{1}{2i}+1}}\Bigg),\quad \forall i\in\mathbb{N},
	\end{equation*}
	where $ T_* $ is the blowup time and $ x_* $ is the blow-up location. Among these solutions, it is known that $ \overline{W}=W_1 $ is the only stable self-similar blow-up profile, while the other profiles $ U_i $ with $ i\geq 2 $ are all unstable \cite{MR4468858}. The profile $ \overline{W} $ solves the steady-state self-similar Burgers equation, in which $\overline{W} $ satisfies
	\begin{equation}\label{eq:2.13}
	-\frac12\overline{W}+\Big(\frac32y+\overline{W}\Big)\partial_{y}\overline{W}=0,
	\end{equation}
	which admits an explicit expression (by some suitable parameterization) 
	\begin{equation*}
	\overline{W}(y)=\Bigg(-\frac{y}{2}+\bigg(\frac{1}{27}+ \frac{y^2}{4}\bigg)^{\frac12}\Bigg)^{\frac13}-\Bigg(\frac{y}{2}+\bigg(\frac{1}{27}+ \frac{y^2}{4}\bigg)^{\frac12}\Bigg)^{\frac13}.
	\end{equation*}

	At the origin, one may calculate to find that
	\begin{equation}\label{eq:2.14}
	\begin{aligned}
	&\overline{W}(0)=0,\quad\partial_y\overline{W}(0)=-1,\quad\partial_y^2\overline{W}(0)=0,\quad\\
	&\partial_y^3\overline{W}(0)=6,\quad\partial_y^{2k}\overline{W}(0)=0,\quad k=2,3,4\ldots.
	\end{aligned}
	\end{equation}
	Let $\langle \cdot \rangle$ be the Japanese bracket, defined as $\langle y \rangle=\sqrt{1+y^2}$. 
	It can be checked that $\overline{W}$ possesses the following bounds:
	\begin{equation}\label{eq:one-order}
	-1\leq \partial_y\overline W\leq 0,
	\end{equation}
	and
	\begin{equation}\label{eq:2.15}
	\begin{aligned}
	&\Vert\overline W\Vert_{L^\infty}\leq  \langle y \rangle^{\frac13},\quad\quad\;\;\Vert\partial_y\overline W\Vert_{L^\infty}\leq  \langle y \rangle^{-\frac23},\\
	&\Vert\partial^2_y\overline W\Vert_{L^\infty}\leq  \langle y \rangle^{-\frac{5}{3}},\quad\Vert\partial^k_y\overline W\Vert_{L^\infty}\lesssim  \langle y \rangle^{\frac{1}{3}-k},\quad k=3,4,5\ldots.
	\end{aligned}
	\end{equation}
	We also need the following sharper bound:
	\begin{equation}\label{eq:2.16}
	\frac{1}{4}\langle y \rangle^{-\frac23} \leq|\partial_{y} \overline{W}(y)| \leq \frac{7}{20}\langle y \rangle^{-\frac23}\quad\text{for}\quad |y| \geq 100 .
	\end{equation}
	
	\subsection{Evolution of equations of higher order derivatives}\label{sec2.3}
	
	\subsubsection {Higher-order derivatives for the $(W, Z)$-system}
	For $n\geq 1$, applying $\partial^n_y$ to \eqref{eq:2.7} yields 
	\begin{equation}\label{eq:2.17}
	\left\{\begin{aligned}
	&\bigg(\partial_s+\frac{3n-1} {2}+\beta_\tau(n+\mathbf{1}_{n\geq 2})\partial_yW\bigg)\partial^n_y W+\mathcal{V}_W\cdot\partial^{n+1}_y W=F^{(n)}_W,\\
	&\bigg(\partial_s+\frac{3}{2}n+\frac{n}{3}\beta_\tau \partial_yW\bigg)\partial^n_y Z+\mathcal{V}_Z\cdot\partial^{n+1}_y Z=F^{(n)}_Z
	\end{aligned}\right.
	\end{equation}
	with the forcing terms 
	\begin{equation}\label{eq:2.18}
	\begin{aligned}
	F^{(n)}_W&=-\frac{3}{4}\beta^{*} \beta_\tau e^{-\frac{s}{2} }\partial_{y}^{n}B-\sum^{n-1}_{k=0}\binom{n} {k}\partial^{n-k}_y G_W\partial^{k+1}_y W\\
	&\quad-\beta_\tau\mathbf{1}_{n\geq 3}\sum^{n-2}_{k=1}\binom{n} {k}\partial^{n-k}_yW\partial^{k+1}_y W,
	\end{aligned}
	\end{equation}
	and
	\begin{equation}\label{eq:2.19}
	\begin{aligned}
	F^{(n)}_Z&=-\frac{3}{4}\beta^{*} \beta_\tau e^{-s}\partial_{y}^{n}B-\sum^{n-1}_{k=0}\binom{n} {k}\partial^{n-k}_y G_Z\partial^{k+1}_y Z\\
	&\quad-\frac{1}{3}\beta_\tau\mathbf{1}_{n\geq 2}\sum^{n-2}_{k=0}\binom{n} {k}\partial^{n-k}_yW\partial^{k+1}_y Z.
	\end{aligned}
	\end{equation}

	\subsubsection {Higher-order derivatives for evolution of $\widetilde W$}
	As a part of the main result, we will show that $W$ converges asymptotically to $\overline{W}$. Therefore, it is
	necessary to work with the equations on $\widetilde W(y,s)=W(y,s)-\overline W(y)$.
	In view of \eqref{eq:2.7} and \eqref{eq:2.13}, one deduces 
	\begin{equation}\label{eq:2.20}
	\begin{aligned}
	&\partial_s\widetilde W+\bigg(\beta_\tau\partial_y\overline W-\frac{1}{2}\bigg)\widetilde W+\mathcal{V}_W\cdot\partial_y\widetilde W=\widetilde F_W
	\end{aligned}
	\end{equation}
	with
	\begin{equation*}
	\begin{aligned}
	\widetilde F_W=-\frac{3}{4}\beta^{*} \beta_{\tau} e^{-\frac{s}{2}} B-e^{-\frac s2}\beta_\tau\dot\kappa-(\beta_{\tau}\dot{\tau}\overline W+G_W)\partial_y\overline W.
	\end{aligned}
	\end{equation*}
	For $n\geq 1$, acting $\partial^n_y $ on \eqref{eq:2.20}, one finds
	\begin{equation}\label{eq:2.21}
	\Big(\partial_s+\frac{3n-1}{2}+\beta_\tau(\partial_y\overline W+n\partial_yW)\Big)\partial^n_y \widetilde W+\mathcal{V}_W\cdot\partial^{n+1}_y \widetilde W=\widetilde F^{(n)}_W,
	\end{equation}
	where the forcing terms $\widetilde F^{(n)}_W$ take the following form:
	\begin{equation*}
	\begin{aligned}
	\widetilde F^{(n)}_W&=\partial^n_y \widetilde F_W-\sum^{n-1}_{k=0}\binom{n} {k}\Big(\partial^{n-k}_y G_W\partial^{k+1}_y \widetilde W+\beta_\tau\partial^{n-k+1}_y\overline{W}\partial^{k}_y \widetilde W\Big)\\
	&\quad-\beta_\tau\mathbf{1}_{n\geq 2}\sum^{n-2}_ {k=0}\binom{n} {k}\partial^{n-k}_yW\partial^{k+1}_y \widetilde W.
	\end{aligned}
	\end{equation*}
	
	\subsection{Constraints on  $W$ and equations of the dynamic modulation variables}\label{sec2.4}
	The wave breaking is characterized by the following constraints on $W$:
	\begin{equation}\label{eq:2.22}
	W(0,s)=0,\quad\partial_yW(0,s)=-1,\quad\partial_y^2W(0,s)=0.
	\end{equation}
	Given these constraints, we can derive the equations for the dynamic modulation variables $(\kappa, \tau, \xi)$.  
	Inserting $y=0$ in \eqref{eq:2.7} gives
	\begin{equation}\label{eq:2.23}
	\dot{\kappa}=\frac1{\beta_\tau}e^{\frac s2}G_W(0,s)-\frac{3}{4}\beta^{*}  B(0,s),
	\end{equation}
	and \eqref{eq:2.17} with $n=1$ yields
	\begin{equation}\label{eq:2.24}
	\dot{\tau}=\frac1{\beta_\tau}\partial_yG_W(0,s)-\frac{3}{4}\beta^{*}  e^{-\frac{s}{2}}\partial_{y}B(0,s).
	\end{equation}
	By the definition of $G_{W}$ in \eqref{eq:2.9}, we have
	\begin{equation}\label{eq:2.25}
	\dot{\xi}=\kappa+\frac{1}{3}Z(0,s)-e^{-\frac{s}{2}}(1-\dot{\tau})G_{W}(0,s).
	\end{equation}
	Conversely, the constraints \eqref{eq:2.22} will also be maintained under the evolution equations \eqref{eq:2.23}-\eqref{eq:2.25}.

	\section{Main results}\label{sec3}
	\subsection{Initial data on self-similar variables}
	Let $M \gg 1$ and $0<\delta=\delta(M) \ll 1 $, both of which will be determined later in the proof. For convenience, we also introduce the parameters
	$$s_{0}=-\log \delta, \quad \ell=(\log M)^{-2}  $$
	and denote
	$$W_{0}(\cdot)=W(\cdot, s_{0}), \quad \widetilde{W}_{0}(\cdot)=\widetilde{W}(\cdot, s_{0}) .$$

	In the following, we select the initial data used in the main result. It suffices to give the initial data in self-similar variables, which can be easily translated into the corresponding ones in the physical variables using \eqref{eq:2.3}-\eqref{eq:2.5}.
	
	First, to meet the constrains \eqref{eq:2.22}, one shall impose the following conditions at $y=0$:
	\begin{equation}\label{eq:3.1}
	W_{0}(0)=0, \quad \partial_{y} W_{0}(0)=\min _{y} \partial_{y} W_{0}=-1, \quad \partial_{y}^{2} W_{0}(0)=0 .
	\end{equation}
	Next, we impose the following conditions in different regions of $ y $ for $\widetilde{W}_{0}$ and $W$:
	\begin{equation}\label{eq:3.2}
	\text{(Zeroth-order derivative)}\begin{cases}
	|\widetilde{W}_{0}(y)| \leq \frac{1}{2} \delta^{\frac{1}{12}} \ell^{4}, & 0 \leq|y| \leq \ell, \\ 
	|\widetilde{W}_{0}(y)| \leq \delta^{\frac{1}{12}}\langle y\rangle^{\frac{1}{3}}, & \ell \leq|y| \leq \frac{1}{2} \delta^{-\frac{3}{2}},
	\end{cases} 
	\end{equation}
	\begin{equation}\label{eq:3.3} 
	\text{(First-order derivative)}\begin{cases}
	|\partial_{y} \widetilde{W}_{0}(y)| \leq \frac{1}{2} \delta^{\frac{1}{12}} \ell^{3}, & 0 \leq|y| \leq \ell, \\ |\partial_{y} \widetilde{W}_{0}(y)| \leq \delta^{\frac{1}{12}}\langle y\rangle^{-\frac{2}{3}}, & \ell \leq|y| \leq \frac{1}{2} \delta^{-\frac{3}{2}}, \\ 
	|\partial_{y} W_{0}(y)| \leq \delta, & \frac{1}{2} \delta^{-\frac{3}{2}} \leq|y|<\infty,
	\end{cases} 
	\end{equation}
	\begin{equation}\label{eq:3.4}
	\text{(Second-order derivative)}\begin{cases}
	|\partial_{y}^{2} \widetilde{W}_{0}(y)| \leq \frac{1}{2} \delta^{\frac{1}{12}} \ell^{2}, & 0 \leq|y| \leq \ell, \\ 
	|\partial_{y}^{2} W_{0}(y)| \leq M^{\frac{1}{10}}, & \ell \leq |y|<\infty,
	\end{cases} 
	\end{equation}
	and
	\begin{equation}\label{eq:3.5}
	\text{(Higher-order derivatives)}\begin{cases}
	|\partial_{y}^{3} \widetilde{W}_{0}(y)| \leq \frac{1}{2} \delta^{\frac{1}{12}} \ell, & 0 \leq|y| \leq \ell, \\
	|\partial_{y}^{4} \widetilde{W}_{0}(y)| \leq \frac{1}{4} \delta^{\frac{1}{12}}, & 0 \leq|y| \leq \ell.
	\end{cases}
	\end{equation}
	Then, we impose the following global  $L^{\infty}$  constraints on $W$:
	\begin{equation}\label{eq:3.6}
	\begin{cases}
	\|\partial_{y}^{4} W_{0}\|_{L^{\infty}} \leq \frac{1}{2} M, \\
	\|W_{0}+\delta^{-\frac{1}{2}} \kappa_{0}\|_{L^{\infty}} \leq \frac{1}{2} M \delta^{-\frac{1}{2}},
	\end{cases}
	\end{equation}
	and the precise third derivative bootstrap at  $y=0$ for $\widetilde{W}_{0}$: 
	\begin{equation}\label{eq:3.7}
	|\partial_{y}^{3} \widetilde{W}_{0}(0)| \leq \frac{1}{4} \delta^{\frac{1}{9}}.
	\end{equation}
	
	For the initial datum of $Z$, we assume
	\begin{equation}\label{eq:3.8}
	\begin{cases}
	\| Z_{0}\|_{L^{\infty}} \leq \frac{1}{2} M \delta, \\ 
	\|\partial_{y} Z_{0}\|_{L^{\infty}} \leq \frac{1}{2} M \delta^{\frac{5}{6}},\\
	\|\partial^{4}_{y} Z_{0}\|_{L^{\infty}} \leq \frac{1}{2} M \delta^{\frac{2}{3}}. 
	\end{cases} 
	\end{equation}

	\subsection{Main results}
	The main result in this paper is stated as follows: 
	\begin{theorem}\label{thm1}
		There exist a sufficiently large  $M>0$  and  a  sufficiently small  $\delta=\delta(M)>0$  such that if the initial data  $(w_{0},z_{0}) \in \dot{H}^{5}(\mathbb{R})$  satisfies \eqref{eq:3.1}-\eqref{eq:3.8} and $b(x)\in\dot{H}^{6}(\mathbb{R})$, then the unique solution $$(w,z) \in C([-\delta, T_{*}), \dot{H}^{5}(\mathbb{R}))\cap C^{1}([-\delta, T_{*}), \dot{H}^{4}(\mathbb{R}))$$ to the Cauchy problem of \eqref{eq:2.2} with $(w_{0},z_{0})$ forms a wave breaking in finite time. Furthermore,
		\begin{itemize}    \item[(\uppercase\expandafter{\romannumeral1})] the blow-up time and location satisfy respectively
			$$T_{*} \leq 2 M\delta^{\frac{4}{3}} \quad\text{and}\quad|x_{*}| \leq 2M\delta.$$
			\item[(\uppercase\expandafter{\romannumeral2})]  The sup-norm of $w$  is bounded:  $\|w(\cdot, t)\|_{L^{\infty}} \leq M$  for all  $t \in[-\delta, T_{*}] $.
			\item[(\uppercase\expandafter{\romannumeral3})]  The sup-norm of $\partial_{x} w$  blows up and has the following blowup rate
			$$\frac{1}{2(T_{*}-t)} \leq\|\partial_{x} w\|_{L^{\infty}} \leq \frac{2}{T_{*}-t} .$$
			\item[(\uppercase\expandafter{\romannumeral4})] The function $w$ displays a cusp singularity at  $(x_{*}, T_{*})$  with  $w(\cdot, T_{*}) \in C^{\frac{1}{3}}(\mathbb{R}) $.
			\item[(\uppercase\expandafter{\romannumeral5})]  $w$  converges asymptotically in the self-similar variables to a stable self-similar solution  $\overline{W}_{\nu} $ of the inviscid Burgers equation, namely
			\begin{equation}\label{eq:3.9}
			\limsup _{s \rightarrow \infty}\|W-\overline{W}_{\nu}\|_{L^{\infty}}=0,
			\end{equation}
			where  $\overline{W}_{\nu}$  is defined by
			\begin{equation}\label{eq:3.10}
			\overline{W}_{\nu}(y)=\Big(\frac{\nu}{6}\Big)^{-\frac{1}{2}} \overline{W}\bigg(\Big(\frac{\nu}{6}\Big)^{\frac{1}{2}} y\bigg)
			\end{equation}
			with $ \nu=\lim _{s \rightarrow \infty} \partial_{y}^{3} W(0, s)$.
		\end{itemize}
	\end{theorem}
	Moreover, the blow-up results in Theorem \ref{thm1} are stable under perturbation, so one can relax the initial conditions above to reach  the following:
	\begin{corollary}\label{coro1}
		There exists an open set of initial data in the $\dot{H}^{5} $ topology satisfying the hypothesis in Theorem \ref{thm1} such that the conclusions of Theorem \ref{thm1} still hold.
	\end{corollary}
	The proof of Corollary \ref{coro1} is quite similar to \cite{MR4612576,MR4321245}, so we will omit it. In the remaining sections, we will focus on the proof of Theorem \ref{thm1}.
	
	\subsection{Sharp lifespan at times of order $\mathcal{O}(1/\max\{\varepsilon,\beta^* \})$}
	We now prove that the lifespan of the classical solutions to the Cauchy problem \eqref{eq:1.1}-\eqref{initial data} under the non-cavitation condition \eqref{eq:non-cavitation} has the sharp times of order $\mathcal{O}(1/\max\{\varepsilon,\beta^* \})$. This scaling arises because the transformations introduce parameter dependence into the Sobolev norm of the transformed initial data $(w_0, z_0)$. Although the blowup time $T^*$ in Theorem \ref{thm1} is constructed independently of $\varepsilon$ and $\beta^*$, the stated order appears when $T^*$ is expressed in terms of the original variables. Since \eqref{eq:transform} and \eqref{eq:transform3} are linear, the scaling of the initial data is straightforward. We only focus on the nonlinear transformation \eqref{eq:transform2}.
    
    Recall the relations for the initial data:
    \begin{equation*}
    u_0=\varepsilon \bar{v}_0, \quad\sigma_{0}=2\sqrt{H+\varepsilon \zeta_{0}-\beta^{*}b} .
    \end{equation*}
Set $f = \varepsilon \zeta_0 - \beta^* b$ and define \(G(f) = \frac{2f}{\sqrt{H+f}+H}\), so that $\sigma_{0}-2H=G(f)$. To estimate $\| \sigma_0 \|_{\dot H^5}= \| G(f) \|_{\dot H^5}$,  we use the following Moser-type estimate for compositions:
\begin{lemma}[\cite{MR1477408}]\label{Moser-type estimate}
  Let \( s \geq 1 \) be an integer, and let \( G \in C^s(\mathbb{R}) \) be a real-valued function with \( G(0) = 0 \). Assume that \( f \in H^s(\mathbb{R}^d) \cap L^\infty(\mathbb{R}^d) \). Then there exists a constant \( C = C(s, d) > 0 \) such that the composition \( G(f) \) belongs to \( H^s(\mathbb{R}^d) \) and satisfies the estimate
\[
\| G(f) \|_{H^s} \leq C \, \| G' \|_{C^{s-1}} \left( 1 + \| f \|_{L^\infty}^{s-1} \right) \| f \|_{H^s}.
\]  
\end{lemma}
Applying Lemma \ref{Moser-type estimate} it with $s=5$ yields
\begin{equation*}
\begin{aligned}
&\| \sigma_0 \|_{\dot H^5} = \| G(f) \|_{\dot H^5}\\
&\leq \max\{\varepsilon,\beta^* \}C\Big(\frac{1}{h_{min}},\| \zeta_0 \|_{\dot H^5} + \| b \|_{\dot H^5} \Big) (\| \zeta_0 \|_{\dot H^5} + \| b \|_{\dot H^5}).
\end{aligned}
\end{equation*}
For $\|u_0\|_{\dot{H}^5}$,  we have the straightforward bound
\[
\|u_0\|_{\dot{H}^5} = \varepsilon \|\bar{v}_0\|_{\dot{H}^5} \leq \max\{\varepsilon,\beta^* \}\|\bar{v}_0\|_{\dot{H}^5}.
\]       
Combining these two estimates above gives
$$\|(w_{0},z_{0})\|_{\dot H^5}\sim \|(u_{0},\sigma_{0})\|_{\dot H^5} \lesssim   \max\{\varepsilon,\beta^* \}(\| \zeta_0 \|_{\dot H^5} +\| \bar v_0 \|_{\dot H^5}+ \| b \|_{\dot H^5})$$    
 According to the the classical theory of symmetrizable hyperbolic system, the lifespan of a solution is inversely proportional to the Sobolev norm of its initial data. Therefore, the maximal existence time  $T^{*}$ satisfies 
$$T^{*}=\mathcal{O}(1/\max\{\varepsilon,\beta^* \}).$$

	\section{Bootstrap assumptions}\label{sec4}
	The proof of Theorem \ref{thm1} is built on a bootstrap argument, which will be made precisely in this section.

	\subsection {Bootstrap assumptions for the self-similar variables.}\label{sec4.1}
	First, we make the following bootstrap assumptions for $\widetilde{W}$ and $W$ in different regimes:
	\begin{equation}\label{eq:4.1}
	\text{(Zeroth-order derivative)}\begin{cases}
	|\widetilde{W}(y, s)| \leq \delta^{\frac{1}{12}} \ell^{4}, & 0 \leq|y| \leq \ell, \\
	|\widetilde{W}(y, s)| \leq  \delta^{\frac{1}{15}}\langle y \rangle^{\frac 13}, & \ell \leq |y| \leq \frac{1}{2}e^{\frac{3}{2}s},
	\end{cases}
	\end{equation}
	\begin{equation}\label{eq:4.2}
	\text{(First-order derivative)}\begin{cases}
	|\partial_{y} \widetilde{W}(y, s)| \leq \delta^{\frac{1}{12}}\ell^{3}, & 0 \leq|y| \leq \ell, \\
	|\partial_{y} \widetilde{W}(y, s)| \leq  \delta^{\frac{1}{18}}\langle y \rangle^{-\frac 23}, & \ell \leq|y| \leq \frac{1}{2}e^{\frac{3}{2}s}, \\
	|\partial_{y} W(y, s)| \leq 2 e^{-s}, & |y|\geq \frac{1}{2}e^{\frac{3}{2}s},
	\end{cases}
	\end{equation}
	\begin{equation}\label{eq:4.3}
	\text{(Second-order derivative)}\begin{cases}
	|\partial_{y}^{2} \widetilde{W}(y, s)| \leq \delta^{\frac{1}{12}} \ell^{2}, & 0 \leq|y| \leq \ell, \\
	|\partial_{y}^{2} W(y, s)| \leq M^{\frac{1}{5}}, &  |y|\geq  \ell,\\
	\end{cases}
	\end{equation}
	and
	\begin{equation}\label{eq:4.4}
	\text{(Higher-order derivatives)}\begin{cases}
	|\partial_{y}^{3} \widetilde{W}(y, s)| \leq \delta^{\frac{1}{12}}\ell, & 0 \leq|y| \leq \ell, \\
	|\partial_{y}^{4} \widetilde{W}(y, s)| \leq \delta^{\frac{1}{12}},  &0  \leq|y| \leq \ell.
	\end{cases}
	\end{equation}
	Next, we assume the global $ L^{\infty}$ bounds on $W$ that
	\begin{equation}\label{eq:4.5}
	\begin{cases}
	\| W+e^{\frac{s}{2}}\kappa\|_{L^{\infty}} \leq Me^{\frac{s}{2}},\\
	\|\partial_{y}^{4} W\|_{L^{\infty}} \leq M .
	\end{cases}
	\end{equation}
	While at $y=0$, we assume for $\widetilde W$ that
	\begin{equation}\label{eq:4.6}
	|\partial^3_y\widetilde W(0, s)|\leq \delta^{\frac{1}{9}}.
	\end{equation}
	
	Then, for $Z$, we assume the following $L^{\infty}$ bound
	\begin{equation}\label{eq:4.7}
	\begin{cases}
	\|Z\|_{L^{\infty}}\leq  M \delta, \\
	\|\partial_y Z\|_{L^{\infty}} \leq  M e^{-\frac{5}{6}s},\\
	\|\partial^{4}_y Z\|_{L^{\infty}} \leq  M e^{-\frac{2}{3}s}.
	\end{cases}
	\end{equation}
	
	\subsection {Bootstrap assumptions for the dynamic modulations variables.}\label{sec4.2}
	For the dynamic modulation variables, we assume that
	\begin{equation}\label{eq:4.8}
	\begin{aligned}
	|\tau(t)|\leq  2M\delta^{\frac{4}{3}},\quad 
	|\dot\tau(t)|\leq  2Me^{-\frac{s}{3} },
	\end{aligned}
	\end{equation}
	and
	\begin{equation}\label{eq:4.9}
	\begin{aligned}
	|\xi(t)|\leq  2M\delta, \quad |\dot\xi(t)|\leq  2M
	\end{aligned}
	\end{equation}
	for all $-\delta<  t<T_*$.

	\subsection{Consequences of the bootstrap assumptions.}\label{sec4.3}
	By the bootstrap assumptions in Subsections \ref{sec4.1}-\ref{sec4.2}, one can derive some direct consequences that will be frequently used later.
	
	(I) \underline{$L^{\infty}$ bound on $ \partial_{y} W $.} We claim that
	\begin{equation}\label{eq:4.10}
	\frac{99}{100} \leq\|\partial_{y} W\|_{L^{\infty}} \leq \frac{101}{100}.
	\end{equation}
	It suffices to show the second inequality in \eqref{eq:4.10}, as the first one can be established through analogous reasoning. The proof involves an analysis of different scales of $|y|$.  
	When  $0 \leq|y| \leq \ell $, one uses $\eqref{eq:4.2}_{1}$ to infer
	\begin{equation*}
	\begin{aligned}
	|\partial_{y} W(y, s)|\leq|\partial_{y} \widetilde{W}(y, s)|+|\partial_{y} \overline{W}(y)|\leq \delta^{\frac{1}{12}} \ell^{3}+1 \leq \delta^{\frac{1}{18}}+1.
	\end{aligned}
	\end{equation*}
	When $ \ell \leq|y| \leq \frac{1}{2} e^{\frac{3 s}{2}} $, it follows from $\eqref{eq:4.2}_{2}$  that
	\begin{equation*}
	\begin{aligned}
	|\partial_{y} W(y, s)|\leq|\partial_{y} \widetilde{W}(y, s)|+|\partial_{y} \overline{W}(y)|\leq(\delta^{\frac{1}{18}}+1)\langle \ell \rangle^{-\frac{2}{3}} \leq\delta^{\frac{1}{18}}+1.
	\end{aligned}
	\end{equation*}
	When  $\frac{1}{2} e^{\frac{3 s}{2}} \leq|y|<\infty $, one may apply  $\eqref{eq:4.2}_{3}$  to deduce
	\begin{equation*}
	|\partial_{y} W(y, s)| \leq 2 e^{-s} \leq 2 \delta \leq \delta^{\frac{1}{18}}+1
	\end{equation*}
	due to  $s \geq s_{0} $. Consequently, we obtain 
	\begin{equation}\label{eq:4.11}
	\|\partial_{y} W\|_{L^{\infty}} \leq 1+\delta^{\frac{1}{18}},
	\end{equation}
	which leads to the desired upper bound by taking $\delta$ sufficiently small.
	
	We note that by using some more delicate decay of $\overline{W}$, one can further show that
	\begin{equation}\label{eq:4.12}
	\|\partial_{y} W(\cdot, s)\|_{L^{\infty}}=1=-\partial_{y} W(0, s)
	\end{equation}
	with the extremum attained uniquely at  $y=0$ (see \cite{MR4321245} for more details).
	
	(II) \underline{$ L^{\infty} $ bounds on  
		$ (\partial_{y}^{j}W, \partial_{y}^{j}Z) $ for  $ j=2,3 $.} From $\eqref{eq:4.3}_{1}$, we have
	\begin{equation*}
	|\partial_{y}^{2} W(y, s)| \leq|\partial_{y}^{2} \widetilde{W}(y, s)|+|\partial_{y}^{2} \overline{W}(y)| \leq \delta^{\frac{1}{12}} \ell^{2}+1\leq M^{\frac{1}{5}}
	\end{equation*}
	for $0 \leq|y| \leq \ell $. Combining this with $\eqref{eq:4.3}_{2}$, we obtain 
	\begin{equation}\label{eq:4.13}
	\|\partial_{y}^{2} W\|_{L^{\infty}} \leq M^{\frac{1}{5}}.
	\end{equation}
	This,  together with $\eqref{eq:4.5}_{2}$ and the Gagliardo-Nirenberg interpolation inequality, yields
	\begin{equation}\label{eq:4.14}
	\|\partial_{y}^{3} W\|_{L^{\infty}} \lesssim\|\partial_{y}^{2} W\|_{L^{\infty}}^{\frac{1}{2}}\|\partial_{y}^{4} W\|_{L^{\infty}}^{\frac{1}{2}} \lesssim M^{\frac{3}{5}}.
	\end{equation}

	Similarly, by \eqref{eq:4.7}, we can deduce that
	\begin{equation}\label{eq:4.15}
	\begin{aligned} 
	&\|\partial_{y}^{2} Z\|_{L^{\infty}} \lesssim\|\partial_{y} Z\|_{L^{\infty}}^{\frac{2}{3}}\|\partial_{y}^{4} Z\|_{L^{\infty}}^{\frac{1}{3}} \lesssim Me^{-\frac{7}{9}s},\\
	&\|\partial_{y}^{3} Z\|_{L^{\infty}} \lesssim\|\partial_{y} Z\|_{L^{\infty}}^{\frac{1}{3}}\|\partial_{y}^{4} Z\|_{L^{\infty}}^{\frac{2}{3}} \lesssim Me^{-\frac{13}{18}s}.
	\end{aligned}
	\end{equation}
	
	(III) \underline{The bounds on  $\beta_{\tau}$ and $\kappa$.} From \eqref{eq:4.5} and \eqref{eq:4.8}, we have
	\begin{equation}\label{eq:4.16}
	\begin{aligned}
	\frac{99}{100}\leq \frac{1}{1-|\dot{\tau}|}\leq| \beta_{\tau}|=\Big|\frac{1}{1-\dot{\tau}}\Big| \leq 1+2|\dot{\tau}| \leq 1+4Me^{-\frac{s}{3}},
	\end{aligned}
	\end{equation}
	and
	\begin{equation}\label{eq:4.17}
	|\kappa(t)|=|w(\xi(t), t)| \leq\|e^{-\frac{s}{2}} W+\kappa\|_{L^{\infty}} \leq M.
	\end{equation}
	
	(IV) \underline{$L^{\infty}$  bound on $W$.}
	It follows from $\eqref{eq:4.5}_{1}$ and $\eqref{eq:4.17}$ that
	\begin{equation}\label{eq:4.18}
	\|W\|_{L^{\infty}}\leq  \|W+e^{\frac{s}{2}}\kappa\|_{L^{\infty}}+\|e^{\frac{s}{2}}\kappa\|_{L^{\infty}}\leq 2Me^{\frac{s}{2}}.
	\end{equation}

	(V) \underline{The point-wise bound on $G_W$.}
	Notice by the mean value theorem that
	\begin{equation*}
	G_W(y,s)=G_W(0,s)+y\cdot\partial_y G_W(y,s).
	\end{equation*}
	
	First, considering the evolution of $\partial_{y}^2 W$ at $y = 0$ in \eqref{eq:2.17}, one has
	\begin{equation*}
	G_W(0,s)=(\partial_{y}^3W(0,s))^{-1}\partial^{2}_y G_W(0,s).
	\end{equation*}
	From \eqref{eq:2.14}, \eqref{eq:4.6}, \eqref{eq:4.13} and \eqref{eq:4.15}, it follows  that
	\begin{equation*}
	\begin{aligned}
	&\partial_y^3 W(0,s)=\partial_y^3\overline{W}(0,s)+\partial_y^3\widetilde{W}(0,s)\geq 6-\delta^{\frac{1}{9}},\\
	&|\partial^{2}_y G_W(0,s)|\lesssim e^{\frac{s}{2}}|\partial^{2}_{y}Z|\leq M e^{-\frac{5}{18} s}.
	\end{aligned}
	\end{equation*}
	Consequently,
	\begin{equation}\label{eq:4.19}
	|G_W(0,s)|\leq M e^{-\frac{5}{18} s}.
	\end{equation}
	
	Next, it is easy to observe that
	\begin{equation}\label{eq:4.20}
	|\partial_yG_W|\lesssim e^{\frac{s}{2}}|\partial_y Z|\lesssim M e^{-\frac{s}{3}}.
	\end{equation}
	
	By using \eqref{eq:4.15}, \eqref{eq:4.20}  and \eqref{eq:4.19}, we obtain
	\begin{equation}\label{eq:4.21}
	|G_W(y,s)|\lesssim|G_W(0,s)|+|y||\partial_yG_W|\lesssim M(e^{-\frac{5}{18}s}+|y|e^{-\frac{s}{3}}).
	\end{equation}
	
	(VI) \underline{$ L^{\infty} $ bounds on  $G_W$ and $ G_Z $.}
	The $L^\infty$-norm of $G_W$ and $G_Z$ can be estimated from \eqref{eq:4.7}, \eqref{eq:4.8} and \eqref{eq:4.17} as follows:
	\begin{equation}\label{eq:4.22}
	\begin{aligned}
	\max \{\|G_W\|_{L^{\infty}},\|G_Z\|_{L^{\infty}}\}\lesssim e^{\frac{s}{2}}(\|Z\|_{L^{\infty}}+|\kappa|+|\dot{\xi}|) \leq 5Me^{\frac{s}{2}}.
	\end{aligned}
	\end{equation}

	\section{Point estimates on dynamic modulation variables}\label{sec6}
	This section closes the bootstrap assumptions in \eqref{eq:4.8} and \eqref{eq:4.9} by showing that these estimates actually hold with strictly better prefactors.
	
	(I) Closing of \eqref{eq:4.8}. 
	We begin by estimating $ \dot{\tau} $. It follows from $\eqref{eq:2.9}$ and the bootstrap assumptions that
	\begin{equation}\label{eq:6.1}
	\begin{aligned}
	|\dot{\tau}| &\leq (1-\dot{\tau}) |\partial_y G_{W}(0,s)|+\frac{3}{4}e^{-\frac{s}{2}}|\beta^{*} B(0,s)|\\
    &\leq (1+2M e^{-\frac{s}{3}})e^{\frac{s}{2}}|\partial_y Z(0,s)|+\frac{3}{4}\beta^{*}  e^{-\frac{s}{2}}\|\partial_x b\|_{L^{\infty}}\\
	&\leq (1+\delta^{\frac{1}{4}})e^{\frac{s}{2}}\times Me^{-\frac{5}{6}s}+Me^{-\frac{s}{2}}\leq  \frac{3}{2}M e^{-\frac{s}{3}}. 
	\end{aligned}
	\end{equation}
	
	To estimate $\tau$, one needs first to bound $T^*$. 	
	Noticing that $ -\log(\tau(t)-t)=s $ and $ \tau(T_*)=T_* $, we have
	\begin{equation*}
	\int_{-\delta}^{T_*}(1-\dot{\tau}(t)) d t=\delta,
	\end{equation*}
	which, together with \eqref{eq:6.1}, implies that
	\begin{equation*}
	T_*+\delta \leq \delta+|\dot{\tau}|(T_*+\delta)\leq \delta+\frac{3}{2}M\delta^{\frac{1}{3}} (T_*+\delta),
	\end{equation*}
	and thus
	\begin{equation}\label{eq:6.2}
	T_*\leq 2M\delta^{\frac{4}{3}}
	\end{equation}
	as long as $ \delta $ is taken small enough. Then one can use \eqref{eq:2.3} and \eqref{eq:6.2} to infer
	\begin{equation}\label{eq:6.3}
	|\tau(t)| \leq|\tau(-\delta)|+\int_{-\delta}^{t}|\dot{\tau}(t^{\prime})| d t^{\prime} \leq \frac{3}{2}M\delta^{\frac{1}{3}}(2 M\delta^{\frac{4}{3}}+\delta) \leq \frac{7}{4}M\delta^{\frac{4}{3}}.
	\end{equation}

	(II) Closing of \eqref{eq:4.9}.  
	Recalling \eqref{eq:2.25} that
	\begin{equation*}
	\dot{\xi}=\kappa+\frac{1}{3}Z(0,s)-e^{-\frac{s}{2}}(1-\dot{\tau})G_{W}(0,s),
	\end{equation*}
	from which it follows that
	\begin{equation}\label{eq:6.4}
	|\dot{\xi}|\leq  M+\frac{1}{3}M \delta+M e^{-\frac{7}{9}s} \leq \frac{3}{2} M,
	\end{equation}
	where one has used \eqref{eq:4.7}, \eqref{eq:4.17} and \eqref{eq:4.19} and taking $ M $ large enough. 
	Integrating \eqref{eq:6.4} from $-\delta$ to $ T_{*}$ gives  
	\begin{equation}\label{eq:6.5}
	|\xi(t)| \leq \frac{3}{2} M(2 M\delta^{\frac{4}{3}}+\delta) \leq \frac{7}{4}M\delta.
	\end{equation}
	
	\section{Bounds on Lagrangian trajectories}\label{sec7}
	In this section, we define the Lagrangian flows associated to the transport velocities and provide some lemmas on the bounds of these Lagrangian trajectories. The proofs are similar to those in \cite{MR4321245, MR4612576}, and we include them here for completeness.

	Let $\Phi_{W}^{y_0} (s)$ and $ \Phi_{Z}^{y_0}(s) $ be  the Lagrangian trajectory of $W$ and $Z$, which are respectively defined by
	\begin{equation}\label{eq:7.1}
	\begin{cases}
	\dfrac{d}{ds}\Phi_W^{y_0}(s)=\mathcal{V}_W\circ\Phi_W^{y_0}(s),\\
	\Phi_W^{y_0}(s_0)=y_0
	\end{cases}
	\end{equation}
	and
	\begin{equation}\label{eq:7.2}
	\begin{cases}
	\dfrac{d}{ds}\Phi_Z^{y_0}(s)=\mathcal{V}_Z\circ\Phi_Z^{y_0}(s),\\
	\Phi_Z^{y_0}(s_0)=y_0.
	\end{cases}
	\end{equation}
	
	\subsection{Upper bound  for $\Phi_W^{y_0}$ and $\Phi_Z^{y_0}$}
	\begin{lemma}\label{lem7.1}
		For all $y_{0} \in \mathbb{R}$, we have
		\begin{equation}\label{eq:7.3}
		\max\big\{|\Phi_W^{y_0}(s)|,|\Phi_Z^{y_0}(s)|\big\}\leq \big(|y_{0}|+8 M\delta^{-\frac{1}{2}}\big)e^{\frac{3}{2}(s-s_0)}.
		\end{equation}
		
	\end{lemma}
	\begin{proof}
		By the definition of \eqref{eq:7.1}, one calculates that
		\begin{equation*}
		\frac{d}{ds}\big(e^{-\frac{3s}{2}}\Phi_{W}^{y_0}(s)\big)= e^{-\frac{3s}{2}}\big(\beta_\tau W\circ\Phi^{y_0}_{W}(s)+G_{W}\big),
		\end{equation*}
		which, together with \eqref{eq:4.5}, \eqref{eq:4.17} and \eqref{eq:4.22}, gives
		\begin{equation*}
		\begin{aligned}
		|\Phi_{W}^{y_0}(s)| & \leq|y_0|e^{\frac{3}{2}(s-s_0)}+e^{\frac{3s}{2}}\int_{s_0}^s\big(\beta_\tau W\circ\Phi_{W}^{y_0}(s^{\prime})+G_{W}\big)e^{-\frac{3s^{\prime}}{2}}ds^{\prime} \\
		& \leq|y_0|e^{\frac{3}{2}(s-s_0)}+e^{\frac{3s}{2}}\int_{s_0}^s\big(2M(1+4Me^{-\frac{s}{3}})+5M\big)e^{\frac{s^{\prime}}{2}}\cdot e^{-\frac{3s^{\prime}}{2}}ds^{\prime} \\
		& \leq(|y_0|+8M\delta^{-\frac{1}{2}})e^{\frac{3}{2}(s-s_0)}.
		\end{aligned}
		\end{equation*}
		The estimate on $\Phi_{Z}^{y_0}(s)$ can be derived analogously.
	\end{proof}
	
	\subsection{Lower bound for $\Phi_W^{y_0}$}
	\begin{lemma}\label{lem7.2}
		Suppose $|y_{0}|\geq  \ell$. Then, the trajectory $\Phi_W^{y_{0}}$ moves away from the origin at an exponential rate satisfying the lower bound
		\begin{equation}\label{eq:7.4}
		\Phi_W^{y_{0}}(s)\geq |y_{0}|e^{\frac{s-s_{0}}{5}}
		\end{equation}
		for all $s\geq  s_0$. Moreover, for sufficiently large $M$,
		\begin{equation}\label{eq:7.5}
		\int_{s_0}^s\langle\Phi_{W}^{y_{0}}(s^{\prime})\rangle^{-\frac{2}{3}}ds'\leq 10\log\frac{1}{\ell}.
		\end{equation}
	\end{lemma}
	\begin{proof}
		We begin with the bootstrap assumption
        \begin{equation}\label{eq:7.6}
		|\Phi_{W}^{y_0}(s)|\geq\frac{1}{2}|y_0|.
		\end{equation}

        By the assumption of this lemma and the definition of $\ell$, we have
        \begin{equation*}
		|y_0|\geq \ell \geq200\delta^{\frac{1}{2}}\geq 200e^{-\frac{1}{2}s}.
		\end{equation*}
		Thus,
		\begin{equation*}
		e^{-\frac{1}{2}s}\leq\frac{1}{100}|\Phi_{W}^{y_0}(s)|.
		\end{equation*}
		Applying the mean value theorem and utilizing  \eqref{eq:4.10}, we find
		\begin{equation*}
		|W(y,s)|\leq|W(0,s)|+\|\partial_{y}W\|_{L^{\infty}}|y|\leq\frac{101}{100}|y|.
		\end{equation*}
		This, together with \eqref{eq:2.8}, \eqref{eq:4.16} and \eqref{eq:4.21}, leads to
		\begin{equation*}
		\begin{aligned}
		|\Phi_{W}^{y_{0}}(s)|\cdot|\mathcal{V}_{W}\circ\Phi_{W}^{y_{0}}(s)| & \geq\frac{3}{2}|\Phi_{W}^{y_{0}}(s)|^{2}-\big(\beta_\tau|W\circ\Phi_{W}^{y_{0}}(s)|+|G_{W}|\big)|\Phi_{W}^{y_{0}}(s)| \\
		& \geq\frac{3}{2}|\Phi_{W}^{y_{0}}(s)|^{2}-(1+4Me^{-\frac{s}{3}})\times\frac{101}{100}|\Phi_{W}^{y_{0}}(s)|^{2}\\
		&\quad-Me^{-\frac{5}{18} s}|\Phi_{W}^{y_{0}}(s)|-Me^{-\frac{s}{3}} |\Phi_{W}^{y_{0}}(s)|^{2}\\
		& \geq\frac{1}{5}|\Phi_{W}^{y_{0}}(s)|^{2}.
		\end{aligned}
		\end{equation*}
		Recalling \eqref{eq:7.2}, it then follows that
		\begin{equation*}
		\frac{d}{ds}|\Phi_{W}^{y_{0}}(s)|^{2}\geq \frac{2}{5}|\Phi_{W}^{y_{0}}(s)|^{2}.
		\end{equation*}
		Solving this inequality gives \eqref{eq:7.4}. This also closes the bootstrap assumption \eqref{eq:7.6} by a better factor.
		
		It remains to show \eqref{eq:7.5}. Using \eqref{eq:7.4}, we have 
		\begin{equation*}
		\begin{aligned}
		\int_{s_{0}}^{s}\langle\Phi_{W}^{y_{0}}(s^{\prime})\rangle^{-\frac{2}{3}} d s^{\prime} & \leq \int_{s_{0}}^{s}\langle|X_{0}| e^{\frac{1}{5}(s-s_{s})}\rangle^{-\frac{2}{3}} d s^{\prime}  \leq \int_{s_{0}}^{s}(1+\ell^{2} e^{\frac{2}{5}(s^{\prime}-s_{0})})^{-\frac{1}{3}} d s^{\prime} \\
		& \leq \int_{s_{0}}^{s_{0}+5 \log \frac{1}{\ell}} 1 d s^{\prime}+\int_{s_{0}+5 \log \frac{1}{\ell}}^{s} \ell^{-\frac{2}{3} }e^{-\frac{2}{15}(s-s_{0})} d s^{\prime} \\
		& \leq 5 \log \frac{1}{\ell}+\frac{15}{2} \leq 10 \log \frac{1}{\ell},
		\end{aligned}
		\end{equation*}
		where the last inequality holds for sufficiently large $M$.
	\end{proof}

	\section{$L^{\infty} $ estimates on $Z$}\label{sec8}
	
	\subsection{Closing $\eqref{eq:4.7}_{1}$}
    Notice that
\begin{equation*}
    \Big|\frac{3}{4}\beta^{*}  \beta_\tau e^{-s} B\Big| \leq \frac{3}{4}\beta^{*}  (1+4Me^{-\frac{s}{3}})\|\partial_{x} b\|_{L^{\infty}}e^{-s}\leq \frac{1}{4}M e^{-s}.
\end{equation*}
	Composing $\eqref{eq:2.7}_{2}$ with the Lagrangian trajectory $\Phi_{Z}^{y_{0}}(s)$, as defined in \eqref{eq:7.2}, and using $\eqref{eq:3.8}_{1}$, one may deduce that
	\begin{equation*}
	|Z\circ \Phi_{Z}^{y_{0}}(s)|\leq |Z(y_{0},s_{0})|+\frac{1}{4}M\int_{s_{0}}^{s}{e^{-s'}}ds'\leq \frac{3}{4}M\delta,
	\end{equation*}
	which closes the bootstrap assumption $\eqref{eq:4.7}_{1}$. 
	\subsection{Closing $\eqref{eq:4.7}_{2}$}
	Define  $V:=e^{\frac{5}{6}s} \partial_{y}Z$. Closing the bootstrap assumption  $\eqref{eq:4.7}_{2}$ reduces to show
	\begin{equation*}
	|V \circ \Phi_{Z}^{y_{0}}(s)| \leq \frac{3}{4} M.
	\end{equation*} 
	A small calculation finds that $V$  is governed by
	\begin{equation}\label{eq:8.1}
	\begin{aligned}
	(\partial_{s}+ \underbrace{\frac{2}{3}+\frac{\beta_{\tau}}{3} \partial_{y} W}_{\mathcal{D}}) V+\mathcal{V}_{Z} \partial_{y} V=\underbrace{-\frac{3}{4}\beta^{*}  \beta_\tau e^{-\frac{s}{6}}\partial_{y} B-e^{\frac{5}{6}s}\partial_{y}G_{Z}\partial_{y} Z}_{\mathcal{F}}.
	\end{aligned}
	\end{equation}
	By using $\eqref{eq:4.11}$, \eqref{eq:4.16} and taking $\delta$ sufficiently small, the damping is bounded below by
	\begin{equation*}
	\mathcal{D} \geq \frac{2}{3}-\frac{1}{3}(1+4Me^{-\frac{s}{3}})(1+\delta^{\frac{1}{18}}) \geq \frac{1}{6},
	\end{equation*}
	which leads to
	\begin{equation*}
	e^{-\int_{s_{0}}^{s} \mathcal{D} \circ \Phi^{y_{0}}_{Z}(s^{\prime}) d s^{\prime}}  \leq e^{-\frac{1}{6}(s-s_0)}.
	\end{equation*}
	On the other hand, the forcing term can be bounded by
	\begin{equation*}
	|\mathcal{F}|\lesssim e^{-\frac{5}{3} s}\|\partial^{2}_{x} b\|_{L^{\infty}}+ e^{\frac{4}{3} s}|\partial_{y} Z|^{2}\leq e^{-\frac{s}{6}}.
	\end{equation*}
	Thus, composing \eqref{eq:8.1} with the Lagrangian trajectory $\Phi_{Z}^{y_{0}}(s)$  and using $\eqref{eq:3.8}_{2}$ yield 
	\begin{equation*}
	\begin{aligned}
	|V\circ \Phi^{y_{0}}_{Z}(s)| & \leq e^{-\frac{1}{6}(s-s_0)} |V \circ \Phi_{Z}^{y_{0}}(s_{0})|+e^{-\frac{1}{6}(s-s_0)}\int_{s_{0}}^{s} e^{-\frac{s^{\prime}}{6}} d s^{\prime} \\
	& \leq e^{-\frac{1}{6}(s-s_0)}(|V \circ\Phi_{Z}^{y_{0}}(s_{0})|+C \delta^{\frac{1}{6}})\leq \frac{3}{4} M^{\frac{3}{4}}.
	\end{aligned}
	\end{equation*}

	\subsection{Closing $\eqref{eq:4.7}_{3}$}
	To close $\eqref{eq:4.7}_{3}$, we begin by
	calculating the evolution of $e^{\frac{2}{3}s}\partial^{4}_{y}Z$ to obtain
	\begin{equation}\label{eq:8.2}
	\begin{aligned}
	&\Big(\partial_s+\underbrace{\frac{16} {3}+\frac{4}{3}\beta_\tau\partial_y W}_{\mathcal{D}}\Big)e^{\frac{2}{3}s}\partial_y^{4}Z+\mathcal{V}_Z\cdot\partial_y(e^{\frac{2}{3}s}\partial_y^{4}Z)\\
	&=\underbrace{-\frac{3}{4}\beta^{*}  \beta_\tau e^{-\frac{s}{3}} \partial^{4}_{y}B}_{\mathcal{F}_{1}}\underbrace{-e^{\frac{2}{3}s}\sum^{3}_{k=0}\binom{4} {k}\partial^{4-k}_y G_Z\partial^{k+1}_y Z}_{\mathcal{F}_{2}}\\
    &\quad\underbrace{-\frac{1}{3}e^{\frac{2}{3}s}\beta_\tau\sum^{2}_{k=0}\binom{4} {k}\partial^{4-k}_y W\partial^{k+1}_y Z}_{\mathcal{F}_{3}}.
	\end{aligned}
	\end{equation}
	Next, we estimate the damping using \eqref{eq:4.11} and \eqref{eq:4.16}, which yields
	\begin{equation*}
	\mathcal{D} \geq \frac{16}{3}-\frac{4}{3}(1+4Me^{-\frac{s}{3}})\frac{101}{100} \geq 1.
	\end{equation*}
	This leads to
	\begin{equation}\label{eq:8.3}
	e^{-\int_{s_{0}}^{s} \mathcal{D} \circ \Phi^{y_{0}}_{Z}(s^{\prime}) d s^{\prime}}  \leq e^{-(s-s_{0})}.
	\end{equation}
	Additionally, the forcing terms are estimated by \eqref{eq:4.5},  \eqref{eq:4.7} and \eqref{eq:4.13}-\eqref{eq:4.15} as follows:
	\begin{equation}\label{eq:8.4}
	\begin{aligned}
    |\mathcal{F}_{1}| & \lesssim e^{-\frac{s}{3}}\|\partial^{4}_{y} B\|_{L^{\infty}}\leq e^{-\frac{19}{3}s}\|\partial^{5}_{x} b\|_{L^{\infty}}\leq \delta,\\
	|\mathcal{F}_{2}| & \lesssim e^{\frac{7}{6}s}(\|\partial_{y} Z\|_{L^{\infty}}\|\partial_{y}^{4} Z\|_{L^{\infty}}+\|\partial_{y}^{2} Z\|_{L^{\infty}}\|\partial_{y}^{3} Z\|_{L^{\infty}})\\
    &\lesssim  M^{2}e^{-\frac{s}{3}}\leq \delta^{\frac{1}{6}},\\
	|\mathcal{F}_{3}| &\lesssim e^{\frac{2}{3}s}\|\partial_{y} Z\|_{L^{\infty}}\|\partial_{y}^{4} W\|_{L^{\infty}}+e^{\frac{2}{3}s}\|\partial_{y}^{2} W\|_{L^{\infty}}\|\partial_{y}^{3} Z\|_{L^{\infty}}\\
	&\quad+e^{\frac{2}{3}s}\|\partial_{y}^{2} Z\|_{L^{\infty}}\|\partial_{y}^{3} W\|_{L^{\infty}}\\
	&\lesssim  M^{2}e^{-\frac{s}{6}}+M^{\frac{6}{5}}e^{-\frac{1}{18}s} +M^{\frac{8}{5}}e^{-\frac{1}{9}s}\leq \delta^{\frac{1}{21}}.
	\end{aligned}
	\end{equation}
	Consequently, it follows from $\eqref{eq:3.8}_{3}$ and \eqref{eq:8.2}-\eqref{eq:8.4} that
	\begin{equation*}
	\begin{aligned}
	&\quad |e^{\frac{2}{3}s}\partial_{y}^{4}Z\circ\Phi_{Z}^{y_{0}}(s)|\\
    & \leq\|\delta^{-\frac{2}{3}}\partial_{y}^{4}Z_{0}\|_{L^{\infty}}e^{-\int_{s_{0}}^{s}\mathcal{D}\circ\Phi_{Z}^{{y}_{0}}(s^{\prime})\:ds^{\prime}}  \\
	&\quad+\int_{s_0}^s|(\mathcal{F}_{1}+\mathcal{F}_{2}+\mathcal{F}_{3})\circ\Phi_{Z}^{y_0}(s')|e^{-\int_{s'}^s\mathcal{D}\circ\Phi_{Z}^{y_0}(s'')\:ds''}\:ds' \\
	&\leq\|\delta^{-\frac{2}{3}}\partial_{y}^{4} Z_{0}\|_{L^{\infty} 
	} e^{-(s-s_{0})}+\delta^{\frac{1}{24}} \int_{s_{0}}^{s} e^{-(s-s^{\prime})} d s^{\prime} \\
	&\leq\|\delta^{-\frac{2}{3}}\partial_{y}^{4} Z_{0}\|_{L^{\infty}} e^{-(s-s_{0})}+\delta^{\frac{1}{24}}(1-e^{-(s-s_{0})}) \\
	&\leq \frac{1}{2} M+\delta^{\frac{1}{24}} \leq \frac{3}{4} M.
	\end{aligned}
	\end{equation*}
	This completes the closure of the bootstrap assumption $\eqref{eq:4.7}_{3}$.
	
	\section{$L^{\infty} $ estimates on $W$}\label{sec9}
	\subsection{Closing of $\eqref{eq:4.5}_{1}$}
	A small calculation finds that $e^{-\frac{s}{2}}W+\kappa$ satisfies
	\begin{equation}\label{eq:9.1}
	\partial_{s}(e^{-\frac{s}{2}}W+\kappa)+\mathcal{V}_{W}\cdot\partial_{y}(e^{-\frac{s}{2}}W+\kappa)=-\frac{3}{4}\beta^{*}  \beta_\tau e^{-s} B.
	\end{equation}
    Notice that
\begin{equation*}
    \Big|-\frac{3}{4}\beta^{*}  \beta_\tau e^{-s} B\Big|\leq \frac{3}{4}\beta^{*} (1+4Me^{-\frac{s}{3}})\|\partial_{x}b\|_{L^{\infty}} e^{-s}\leq e^{-\frac{s}{2}},
 \end{equation*}
	we then compose \eqref{eq:9.1} with the Lagrangian trajectory $\Phi_{W}^{y_{0}}(s)$, as defined in \eqref{eq:7.1}, and use$ \eqref{eq:3.6}_{3}$ to obtain
	\begin{equation*}
	\begin{aligned}
	\big|(e^{-\frac{s}{2}} W+\kappa) \circ \Phi_{W}^{y_{0}}(s)\big|  &\leq \big|e^{-\frac{s_{0}}{2}} W(y_{0}, s_{0})+\kappa(-\delta)\big| +\int_{s_{0}}^{s} e^{-\frac{s^{\prime}}{2}} d s^{\prime}\\
    &\leq \frac{1}{2} M.
	\end{aligned}
	\end{equation*}

	\subsection{Closing of $\eqref{eq:4.5}_{2}$}
	Recall $\eqref{eq:2.17}_{1}$ with $ n=4$:
	\begin{equation}\label{eq:9.3}
	\begin{aligned}
	&\bigg(\partial_{s}+\underbrace{\frac{11}{2}+5 \beta_{\tau} \partial_{y} W}_{\mathcal{D}}\bigg) \partial_{y}^{4} W+\mathcal{V}_{W}\cdot \partial_{y}^{5} W \\
	&=\underbrace{-\frac{3}{4}\beta^{*} \beta_\tau e^{-\frac{s}{2}} \partial^{4}_{y}B}_{\mathcal{F}_{1}}\underbrace{-\sum_{0\leq  k\leq  3}\binom{4}{k}\partial^{4-k}_y G_W\partial^{k+1}_y W}_{\mathcal{F}_{2}}\underbrace{-10 \beta_{\tau} \partial_{y}^{2} W \partial_{y}^{3} W}_{\mathcal{F}_{3}} .
	\end{aligned}
	\end{equation}
	By \eqref{eq:4.10} and \eqref{eq:4.16}, one deduces
	\begin{equation*}
	\mathcal{D} \geq \frac{11}{2}-5(1+4Me^{-\frac{s}{3}})\frac{101}{100} \geq \frac{1}{4},
	\end{equation*}
	which gives
	\begin{equation}\label{eq:9.4}
	e^{-\int_{s_{0}}^{s} \mathcal{D} \circ \Phi^{y_{0}}_{W}(s^{\prime}) d s^{\prime}}  \leq e^{-\frac{1}{4}(s-s_{0})} .
	\end{equation}
	On the other hand, by applying \eqref{eq:4.5},  \eqref{eq:4.7}, and \eqref{eq:4.13}-\eqref{eq:4.15}, one can infer that
	\begin{equation}\label{eq:9.5}
	\begin{aligned}
    |\mathcal{F}_{1}| &\lesssim e^{-\frac{s}{2}}\|\partial_{y}^{4} B\|_{L^{\infty}}\leq e^{-\frac{11}{2}s}\|\partial_{x}^{5} b\|_{L^{\infty}} 
	\leq \delta^{\frac{1}{9}}\\
	|\mathcal{F}_{2}| & \lesssim e^{\frac{s}{2}}\|\partial_{y} W\|_{L^{\infty}}\|\partial_{y}^{4} Z\|_{L^{\infty}}+e^{\frac{s}{2}}\|\partial_{y}^{2} W\|_{L^{\infty}}\|\partial_{y}^{3} Z\|_{L^{\infty}} \\
	&\quad +e^{\frac{s}{2}}\|\partial_{y}^{3} W\|_{L^{\infty}}\|\partial_{y}^{2} Z\|_{L^{\infty}}+e^{\frac{s}{2}}\|\partial_{y}^{4} W\|_{L^{\infty}}\|\partial_{y} Z\|_{L^{\infty}}\\
	& \lesssim  Me^{-\frac{s}{6}}+M^{\frac{6}{5}}e^{-\frac{2}{9}s}+M^{\frac{8}{5}}e^{-\frac{5}{18}s}+M^{2}e^{-\frac{s}{3}}\leq \delta^{\frac{1}{9}},\\
	|\mathcal{F}_{3}| &\lesssim \|\partial_{y}^{2} W\|_{L^{\infty}}\|\partial_{y}^{3} W\|_{L^{\infty}}
	\leq C M^{\frac{4}{5}},
	\end{aligned}
	\end{equation}
	Therefore, it follows from $\eqref{eq:3.6}_{3}$ and \eqref{eq:9.3}-\eqref{eq:9.5} that
	\begin{equation*}
	\begin{aligned}
	|\partial_{y}^{4}W\circ\Phi_{W}^{y_{0}}(s)|& \leq\|\partial_{y}^{4}W_{0}\|_{L^{\infty}}e^{-\int_{s_{0}}^{s}\mathcal{D}\circ\Phi_{W}^{y_{0}}(s^{\prime})\:ds^{\prime}}  \\
	&\quad+\int_{s_0}^s|(\mathcal{F}_{1}+\mathcal{F}_{2}+\mathcal{F}_{3})\circ\Phi_{W}^{y_0}(s')|e^{-\int_{s'}^s\mathcal{D}\circ\Phi_{W}^{y_0}(s'')\:ds''}\:ds' \\
	&\leq\|\partial_{y}^{4} W_{0}\|_{L^{\infty} 
	} e^{-\frac{1}{4}(s-s_{0})}+(2\delta^{\frac{1}{9}}+C M^{\frac{4}{5}}) \int_{s_{0}}^{s} e^{-\frac{1}{4}(s-s^{\prime})} d s^{\prime} \\
	&\leq\|\partial_{y}^{4} W_{0}\|_{L^{\infty}} e^{-\frac{1}{4}(s-s_{0})}+4(2\delta^{\frac{1}{9}}+C M^{\frac{4}{5}})(1-e^{-\frac{1}{4}(s-s_{0})}) \\
	&\leq \frac{1}{2} M+C M^{\frac{4}{5}} \leq \frac{3}{4} M.
	\end{aligned}
	\end{equation*}

	\subsection{Near field  ($0 \leq|y| \leq \ell$)}
	
	\subsubsection{Closing of $\eqref{eq:4.4}_{2}$} Recall that \eqref{eq:2.21} with $n=4$ reads as
	\begin{equation}\label{eq:9.6}
	\bigg(\partial_{s}+\frac{11}{2}+ \beta_{\tau} (\partial_{y} \overline{W}+4\partial_{y} W)\bigg) \partial_{y}^{4} \widetilde{W}+\mathcal{V }_{W}\cdot\partial_{y}^{5} \widetilde{W}=\widetilde F^{(4)}_W
	\end{equation}
	with
	\begin{equation*}
	\begin{aligned}
	\widetilde F^{(4)}_W&=\underbrace{\partial^4_y \widetilde F_W}_{\mathcal{F}_1}\underbrace{-\sum^{3}_{k=0}\binom{4}{k}\partial^{4-k}_y G_W\partial^{k+1}_y \widetilde W}_{\mathcal{F}_2}\underbrace{-\beta_\tau\sum^{3}_{k=0}\binom{4}{k}\partial^{5-k}_y\overline{W}\partial^{k}_y \widetilde W}_{\mathcal{F}_3}\\
	&\quad\;\underbrace{-\beta_\tau\sum^{2}_{k=0}\binom{4} {k}\partial^{4-k}_yW\partial^{k+1}_y \widetilde W}_{\mathcal{F}_4}.
	\end{aligned}
	\end{equation*}
	For the damping, one uses \eqref{eq:4.10} and \eqref{eq:4.16} to estimate
	\begin{equation*}
	\mathcal{D} \geq \frac{11}{2}-(1+4Me^{-\frac{s}{3}})\bigg(1+4\times \frac{101}{100}\bigg) \geq \frac{1}{4},
	\end{equation*}
	yielding
	\begin{equation}\label{eq:9.7}
	e^{-\int_{s_{0}}^{s} \mathcal{D} \circ \Phi^{y_{0}}_{W}(s^{\prime}) d s^{\prime}}  \leq e^{-\frac{1}{4}(s-s_{0})} .
	\end{equation}
	For $|y|\leq \ell$, we use \eqref{eq:2.15}, \eqref{eq:4.1}-\eqref{eq:4.8} and \eqref{eq:4.15} to estimate the forcing term:
	\begin{equation}\label{eq:9.8}
	\begin{aligned}
	|\mathcal{F}_{1}|& \lesssim \Big|\partial_y^{4}\big[(\beta_{\tau}\dot{\tau}\overline{W}+G_{W})\partial_{y}\overline{W}\big]\Big|+\Big|\frac{3}{4}\beta^{*} e^{-\frac{s}{2}}\partial^{4}_{y}B\Big|\\
	&\lesssim|\dot{\tau}|\sum^{4}_{k=0}|\partial_y^{4-k}\overline{W}||\partial_y^{k+1}\overline{W}|+\sum^{4}_{k=0}|\partial_y^{4-k}G_W||\partial^{k+1}_y\overline{W}|+e^{-\frac{13}{2}s}|\partial^{5}_{x}b|\\
	&\lesssim 2Me^{-\frac{s}{3}}\langle y \rangle^{-\frac{17}{3}}+Me^{-\frac{5}{18}s}\langle y \rangle^{-\frac{14}{3}}+ Me^{-\frac{s}{6}}+ e^{-\frac{s}{2}}\leq\delta^{\frac{1}{9}},\\
	|\mathcal{F}_{2}|& \lesssim \sum^{3}_{k=0}\big|\partial^{4-k}_y G_W\big|\big|\partial^{k+1}_y \widetilde W\big|\lesssim \sum^{3}_{k=0}e^{\frac{s}{2}}\big|\partial^{4-k}_y Z\big|\big|\partial^{k+1}_y \widetilde W\big|\\
	&\lesssim M\delta^{\frac{1}{12}}(e^{-\frac{s}{3}}+e^{-\frac{5}{18}s}\ell+e^{-\frac{2}{9}s}\ell^{2}+e^{-\frac{s}{6}}\ell^{3})\leq \delta^{\frac{1}{9}},\\
	|\mathcal{F}_{3}|& \lesssim \sum^{3}_{k=0}\big|\partial^{5-k}_y \overline{W}\big|\big|\partial^{k}_y \widetilde W\big|
	\lesssim\delta^{\frac{1}{12}}\ell^{4}\langle y \rangle^{-\frac{14}{3}}+\delta^{\frac{1}{12}}\ell^{3}\langle y \rangle^{-\frac{11}{3}}\\
	&\quad+\delta^{\frac{1}{12}}\ell^{2}\langle y \rangle^{-\frac{8}{3}}+\delta^{\frac{1}{12}}\ell\langle y \rangle^{-\frac{5}{3}}\leq 5\delta^{\frac{1}{12}}\ell,\\
	|\mathcal{F}_{4}|& \lesssim \sum^{2}_{k=0}\big|\partial^{4-k}_y W\big|\big|\partial^{k+1}_y \widetilde W\big|\lesssim \big(\langle y \rangle^{-\frac{11}{3}}+\delta^{\frac{1}{12}}\big)\delta^{\frac{1}{12}}\ell^{3}\\
	&\quad+\big(\langle y \rangle^{-\frac{8}{3}}+\delta^{\frac{1}{12}}\ell\big)\delta^{\frac{1}{12}}\ell^{2}+\big(\langle y \rangle^{-\frac{5}{3}}+\delta^{\frac{1}{12}}\ell^{2}\big)\delta^{\frac{1}{12}}\ell\leq 5\delta^{\frac{1}{12}}\ell.
	\end{aligned}
	\end{equation}
	Therefore, combining $\eqref{eq:3.5}_{2}$(note that $|y_0|\leq \ell$ ensures the trajectory stays within the region $ |y|\leq \ell $) with \eqref{eq:9.6}-\eqref{eq:9.8}, we apply Gronwall’s inequality to get
	\begin{equation*}
	\begin{aligned}
	|\partial_{y}^{4} \widetilde{W} \circ \Phi_{W}^{y_{0}}(s)| \leq & |\partial_{y}^{4} \widetilde{W}(y_{0}, s_{0})| e^{-\frac{1}{4}(s-s_{0})} \\
	& +(2\delta^{\frac{1}{9}}+10 \delta^{\frac{1}{12}}\ell) \int_{s_{0}}^{s} e^{-\frac{1}{4}(s-s^{\prime})} d s^{\prime} \\
	\leq & \frac{1}{4} \delta^{\frac{1}{12}}+4(2\delta^{\frac{1}{9}}+10 \delta^{\frac{1}{12}}\ell) \leq \frac{1}{2} \delta^{\frac{1}{12}}.
	\end{aligned}
	\end{equation*}
	
	\subsubsection{Closing of \eqref{eq:4.6}} Plugging $ y=0$  into \eqref{eq:2.21} and using \eqref{eq:2.22} lead to 
	\begin{equation*}
	\partial_{s}(\partial^{3}_y\widetilde{W})(0, s)=\widetilde{F}_{W}^{(3)}(0, s)
	-G_{W}(\partial_{y}^{4}\widetilde{W})(0, s)+4\beta_{\tau}\dot{\tau}
	(\partial^{3}_y\widetilde{W})(0, s).
	\end{equation*}
	In view of $\eqref{eq:4.6}$, \eqref{eq:4.8}, \eqref{eq:4.15} and \eqref{eq:4.20}, one can estimate
	\begin{equation*}
	\begin{aligned}
	|\widetilde{F}_{W}^{(3)}(0, s)|&\lesssim |\dot{\tau}|+|(\partial_y^3G_W)(0, s)|+| e^{-\frac{s}{2}}\partial^{3}_{y}B(0,s)|\\
     &\quad +\big(1+|(\partial_y^3\widetilde{W})(0, s)|\big)|(\partial_y G_W)(0, s)| \\
	&\lesssim Me^{-\frac{s}{3}}+Me^{-\frac{2}{9}s} +e^{-\frac{s}{2}} \leq e^{-\frac{7}{36}s}, \\
	\end{aligned}
	\end{equation*}
	which, together with \eqref{eq:4.4}, \eqref{eq:4.8} and \eqref{eq:4.19}, implies that
	\begin{equation}\label{eq:9.9}
	|\partial_{s}(\partial^{3}_y\widetilde{W})(0, s)|\lesssim e^{-\frac{7}{36}s}+M\delta^{\frac{1}{12}}e^{-\frac{5}{18} s}+M\delta^{\frac{1}{12}}\ell e^{-\frac{s}{3}}
	\leq e^{-\frac{s}{6}}.
	\end{equation}
	By \eqref{eq:3.7} and \eqref{eq:9.9}, one deduces 
	\begin{equation}\label{eq:9.10}
	\begin{aligned}
	|\partial_{y}^{3} \widetilde{W}(0,s)| & \leq|\partial_{y}^{3} \widetilde{W}_0(0)|+C \int_{s_{0}}^{s}|\partial_{s}(\partial^{3}_y\widetilde{W})(0, s)| d s^{\prime} \\
	& \leq \frac{1}{4} \delta^{\frac{1}{9}}+C \delta^{\frac{1}{6}}(1-e^{-\frac{1}{6}(s-s_{0})}) \leq \frac{1}{2} \delta^{\frac{1}{9}}.
	\end{aligned}
	\end{equation}

	\subsubsection{Closing of $\eqref{eq:4.1}_{1} $, $\eqref{eq:4.2}_{1}$, $\eqref{eq:4.3}_{1}$ and  $\eqref{eq:4.4}_{1}$} It follows from \eqref{eq:9.10} that
	\begin{equation}\label{eq:9.11}
	\begin{aligned}
	|\partial_{y}^{3} \widetilde{W}(y, s)| & \leq|\partial_{y}^{3} \widetilde{W}(0, s)|+\int_{0}^{y}|\partial_{y}^{4} \widetilde{W}(y^{\prime}, s)| d y^{\prime} \\
	& \leq   \frac{1}{2}\delta^{\frac{1}{9}}+\frac{1}{2}\delta^{\frac{1}{12}}\ell \leq \frac{7}{8} \delta^{\frac{1}{12}}\ell,
	\end{aligned}
	\end{equation}
	which closes the bootstrap assumption $\eqref{eq:4.4}_{1}$. 
	
	One notices that the constraints \eqref{eq:2.22} imply
	\begin{equation*}
	\partial_{y}^{j} \widetilde{W}(0, s)=0, \quad j=0,1,2,
	\end{equation*}
	and then applies \eqref{eq:9.11} to deduce
	\begin{equation*}
	\begin{aligned}
	|\partial_{y}^{2} \widetilde{W}(y, s)| &\leq \int_{0}^{y}|\partial_{y}^{3} \widetilde{W}(y^{\prime}, s)| d y^{\prime} \leq \frac{7}{8} \delta^{\frac{1}{12}} \ell^{2}, \\
	|\partial_{y} \widetilde{W}(y, s)| &\leq \int_{0}^{y}|\partial_{y}^{2} \widetilde{W}(y^{\prime}, s)| d y^{\prime} \leq \frac{7}{8} \delta^{\frac{1}{12}} \ell^{3}, \\
	|\widetilde{W}(y, s)| &\leq \int_{0}^{y}|\partial_{y} \widetilde{W}(y^{\prime}, s)| d y^{\prime} \leq \frac{7}{8} \delta^{\frac{1}{12}} \ell^{4},
	\end{aligned}
	\end{equation*}
	which closes the bootstrap assumptions 
	$\eqref{eq:4.1}_{1}$, $\eqref{eq:4.2}_{1}$ and $\eqref{eq:4.3}_{1}$.

	\subsection{ Middle field ($\ell \leq|y| \leq \frac{1}{2} e^{\frac{3 s}{2}}$)}
	
	\subsubsection{Closing of $\eqref{eq:4.1}_{2}$} Define  $V:=\langle y \rangle^{-\frac{1}{3}} \widetilde{W}$. To close  $\eqref{eq:4.1}_{2}$, it suffices to show
	\begin{equation}\label{eq:9.12}
	|V \circ \Phi_{W}^{y_{0}}(s)| \leq \frac{3}{4} \delta^{\frac{1}{15}}.
	\end{equation}
	
	A direct calculation finds that $V$  is governed by
	\begin{equation}\label{eq:9.13}
	\begin{aligned}
	\bigg(\partial_{s} \underbrace{-\frac{1}{2}+\beta_{\tau} \partial_{y} \overline{W}+\frac{y}{3\langle y \rangle^{2}} \mathcal{V}_{W}}_{\mathcal{D}}\bigg) V+\mathcal{V}_{W}\cdot \partial_{y} V=\langle y \rangle^{-\frac{1}{3}} \widetilde{F}_{W}
	\end{aligned}
	\end{equation}
	with
	\begin{equation*}
	\widetilde{F}_{W}=-\frac{3}{4}\beta^{*}\beta_{\tau} e^{-\frac{s}{2}}B-e^{-\frac s2}\beta_\tau\dot\kappa-(\beta_{\tau}\dot{\tau}\overline W+G_W)\partial_y\overline W .
	\end{equation*}
	According to \eqref{eq:2.15},  $\eqref{eq:4.1}_2$, \eqref{eq:4.20} and \eqref{eq:4.21}, the damping is bounded by 
	\begin{equation*}\label{eq:9.19}
	\begin{aligned}
	\mathcal{D}= & -\frac{1}{2}+\beta_{\tau} \partial_{y} \overline{W}+\frac{y}{3\langle y \rangle^{2}}\bigg(\frac{3}{2} y+\beta_{\tau} W+G_{W}\bigg) \\
	\geq & -\frac{1}{2}-(1+4Me^{-\frac{s}{3}})\langle y \rangle^{-\frac{2}{3}}+\frac{y^{2}}{2\langle y \rangle^{2}} \\
	& -\frac{1}{3}(1+4Me^{-\frac{s}{3}})(1+\delta^{\frac{1}{15}}) y\cdot\langle y \rangle^{-\frac{5}{3}}\\
	&-\frac{|y|}{3\langle y \rangle^{2}} M e^{-\frac{5}{18}s} -\frac{y^{2}}{3\langle y \rangle^{2}} M  e^{-\frac{s}{3}}\\
	\geq & -3\langle y \rangle^{-\frac{2}{3}} -e^{-\frac{s}{6}},
	\end{aligned}
	\end{equation*}
	which, together with \eqref{eq:7.5}, gives
	\begin{equation}\label{eq:9.14}
	e^{-\int_{s_{*}}^{s} \mathcal{D} \circ \Phi_{0}^{y_{0}}(s^{\prime}) d s^{\prime}} \leq e^{\frac{1}{6}\delta}e^{30 \log (\frac{1}{h})}=\frac{5}{4}\ell^{-30}.
	\end{equation}
	For the  forcing term, we can estimate it as follows:
	\begin{equation}\label{eq:9.21}
	\langle y \rangle^{-\frac{1}{3}}| \widetilde{F}_{W}| \lesssim Me^{-\frac{s}{2}}+Me^{-\frac{s}{3}}+Me^{-\frac{5}{18}s}+Me^{-\frac{s}{3}} |y|\cdot\langle y \rangle^{-1}\leq  e^{-\frac{1}{9}s},
	\end{equation}
	where \eqref{eq:2.15}, \eqref{eq:2.23}, \eqref{eq:4.1}, \eqref{eq:4.8} and \eqref{eq:4.21} have been used.

	It follows from \eqref{eq:9.13} and \eqref{eq:9.21} that
	\begin{equation*}
	\begin{aligned}
	|{V} \circ \Phi_{W}^{y_{0}}(s)| & \leq \frac{5}{4}\ell^{-30}|W \circ\Phi_{W}^{y_{0}}(s_{*})|+ \frac{5}{4}\ell^{-30} \int_{s_{e}}^{s} e^{-\frac{s^{\prime}}{9}} d s^{\prime} \\
	& \leq \frac{5}{4}\ell^{-30}(|W \circ \Phi_{W}^{y_{0}}(s_{*})|+C \delta^{\frac{1}{9}}).
	\end{aligned}
	\end{equation*}
	Here $s_{*}$ denotes the first time that the Lagrangian trajectory enters the region $\ell \leq |y| \leq \frac{1}{2} e^{\frac{3 s}{2}}$.
	There are only two cases occurring:\\
	(I) $ \ell \leq|y_{0}| \leq \frac{1}{2} e^{\frac{3 s}{2}}$  and $ s_{*}=s_{0} $. In this case, one shall use $\eqref{eq:3.2}_{2}$  to get
	\begin{equation*}
	|{V} \circ \Phi_{W}^{y_{0}}(s)| \leq \frac{5}{4}\ell^{-30}(\delta^{\frac{1}{12}}+C \delta^{\frac{1}{9}}) \leq \frac{3}{4} \delta^{\frac{1}{15}}.
	\end{equation*}
	(II)  $s_{*}>s_{0}$  and  $|y_{0}|=\ell $. In the case, we apply $\eqref{eq:4.1}_{1} $ to obtain
	\begin{equation*}
	|{V} \circ \Phi_{W}^{y_{0}}(s)| \leq \frac{5}{4}\ell^{-30}\bigg(\frac{1}{2}\langle \ell \rangle^{\frac{1}{3}}\delta^{\frac{1}{12}}\ell^{4}+C \delta^{\frac{1}{9}}\bigg) \leq \frac{3}{4} \delta^{\frac{1}{15}}.
	\end{equation*}
	Collecting the above two cases yields \eqref{eq:9.12}.
	
	\subsubsection{Closing of $\eqref{eq:4.2}_{2} $} Setting  $V:=\langle y \rangle^{\frac{2}{3}} \partial_{y} \widetilde{W} $, in order to close $\eqref{eq:4.2}_{2} $, it only needs to verify
	\begin{equation*}
	|V \circ \Phi_{W}^{y_{0}}(s)| \leq \frac{3}{4}  \delta^{\frac{1}{18}}.
	\end{equation*}
	It is elementary to calculate that
	\begin{equation*}\label{eq:9.22}
	\begin{aligned}
	\bigg(\partial_{s}+\underbrace{1+\beta_{\tau}(\partial_{y} \widetilde{W}+2 \partial_{y} \overline{W})-\frac{2}{3} \frac{y}{\langle y \rangle^{2}} \mathcal{V}_{W}}_{\mathcal{D}}\bigg) V+\mathcal{V}_{W}\cdot \partial_{y} V =\langle y \rangle^{\frac{2}{3}}\widetilde{F}^{(1)}_{W}
	\end{aligned}
	\end{equation*}
	with
	\begin{equation*}
	\widetilde{F}^{(1)}_{W}=\partial_{y}\widetilde{F}_{W}-\partial_{y}G_{W}\partial_{y}\widetilde{W}-\beta_{\tau}\partial_{y}^{2}\overline{W}\widetilde{W}.
	\end{equation*}
	The damping can be bounded below by
	\begin{equation*}
	\begin{aligned}
	\mathcal{D} \geq & 1-(1+4Me^{-\frac{s}{3}})(\delta^{\frac{1}{18}}+2)\langle y \rangle^{-\frac{2}{3}}-\frac{y^{2}}{\langle y \rangle^{2}} \\
	& -\frac{1}{3} \times (1+4Me^{-\frac{s}{3}})(1+\delta^{\frac{1}{15}})  y\cdot \langle y \rangle^{-\frac{5}{3}}\\
	&-\frac{2|y|}{3\langle y \rangle^{2}}\times Me^{-\frac{5}{18}s}-\frac{2y^2}{3\langle y \rangle^{2}}\times Me^{-\frac{s}{3}} \\
	\geq & -3\langle y \rangle^{-\frac{2}{3}}-e^{-\frac{s}{6}},
	\end{aligned}
	\end{equation*} 
	which yields the same estimate as given in \eqref{eq:9.14}.
	
    In the middle field $\ell \leq|y| \leq \frac{1}{2} e^{\frac{3 s}{2}}$, we have $\langle y \rangle^{\frac{2}{3}} \lesssim e^{s}$. For the forcing term, we use the decay of  $\partial_{y}^{j} \overline{W}(j=1,2)$  in \eqref{eq:2.15}, \eqref{eq:4.1}, \eqref{eq:4.2}, \eqref{eq:4.8}, \eqref{eq:4.15} and \eqref{eq:4.21} to estimate
	\begin{equation*}
	\begin{aligned}
	\langle y \rangle^{\frac{2}{3}}|\widetilde{F}^{(1)}_{W}| & \lesssim \langle y \rangle^{\frac{2}{3}}\big(\dot{\tau}(\partial_{y}\overline{W})^{2}+ \dot{\tau}\overline{W}\partial^{2}_{y}\overline{W}+|\partial_{y} G_{W}||\partial_{y}\overline{W}|\\	&\quad+|G_{W}||\partial^{2}_{y}\overline{W}|+|\partial_{y} G_{W}||\partial_{y}\widetilde{W}|+|\partial^{2}_{y}\overline{W}||\widetilde{W}|+e^{-\frac{s}{2}}|\partial_{y} B|\big)\\
	&\lesssim Me^{-\frac{4}{3} s}\langle y \rangle^{\frac{2}{3}}+Me^{-\frac{s}{3}}+Me^{-\frac{5}{18}s}\\
	&\quad+Me^{-s}|y|\langle y \rangle^{-1}+\delta^{\frac{1}{18}}e^{-\frac{s}{3}}+\delta^{\frac{1}{15}}\langle y \rangle^{-\frac{2}{3}}+e^{-2s}\langle y \rangle^{\frac{2}{3}}|\partial^{2}_{x}b|\\
	& \leq  e^{-\frac{s}{6}}++\delta^{\frac{1}{15}}\langle y \rangle^{-\frac{2}{3}}.
	\end{aligned}
	\end{equation*}
    Note that the bottom topography term here has a decay factor of $e^{-\frac{s}{2}}$. Its product with weight $\langle y \rangle^{\frac{2}{3}}\sim e^{s}$ would grow like $e^{\frac{s}{2}}$ in this region. However, applying the chain rule 
    \begin{equation*}
        \frac{d}{dy}=e^{-\frac{3}{2}s}\frac{d}{dx},
    \end{equation*}
    together with the regularity condition $b\in \dot{H}^{6}$, provides an additional decay that is strong enough, leading to the desired decay.

	By summing all of the above estimates, we obtain
	\begin{equation*}
	\begin{aligned}
	|{V} \circ \Phi_{W}^{y_{0}}(s)| &\leq  \frac{5}{4}\ell^{-30}|W \circ \Phi_{W}^{y_{0}}(s_{*})|+ \frac{5}{4}\ell^{-30}\int_{s_{*}}^{s} e^{-\frac{s^{\prime}}{6}} d s^{\prime} \\
	&\quad +\frac{5}{4}\ell^{-30}\delta^{\frac{1}{15}}\int_{s_{*}}^{s}\langle \Phi_{W}^{y_{0}}(s^{\prime}) \rangle^{-\frac{2}{3}} d s^{\prime} \\
	&\leq  \frac{5}{4}\ell^{-30}\big(|W \circ \Phi_{W}^{y_{0}}(s_{*}) |+C \delta^{\frac{1}{6}}\big)+15 \delta^{\frac{1}{15}} \ell^{-30} \log \frac{1}{\ell} .
	\end{aligned}
	\end{equation*}
	Again, we have two cases to consider:\\
	(I) If  $\ell \leq|y_{0}| \leq \frac{1}{2} e^{\frac{3 s}{2}}$  and $ s_{*}=s_{0} $, then it follows from  $\eqref{eq:3.3}_{2}$  that
	\begin{equation*}
	|{V} \circ \Phi_{W}^{y_{0}}(s)| \leq  \frac{5}{4}\ell^{-30}\big(\delta^{\frac{1}{12}}+C  \delta^{\frac{1}{6}}\big)+15 \ell^{-30}  \delta^{\frac{1}{15}}\log \frac{1}{\ell} \leq \frac{3}{4}  \delta^{\frac{1}{18}}.
	\end{equation*}
	(II) If  $s_{*}>s_{0}$  and  $|y_{0}|=\ell $, then one instead uses  $\eqref{eq:4.2}_{1}$  to find
	\begin{equation*}
	|{V} \circ \Phi_{W}^{y_{0}}(s)| \leq  \frac{5}{4}\ell^{-30}\bigg(\frac{1}{2}\langle \ell \rangle^{\frac{2}{3}} \delta^{\frac{1}{12}} \ell^{3}+C \delta^{\frac{1}{6}}\bigg)+15 \delta^{\frac{1}{15}}\ell^{-30} \log \frac{1}{\ell} \leq \frac{3}{4}  \delta^{\frac{1}{18}}.
	\end{equation*}

	\subsection{Far field ($\frac{1}{2}e^{\frac{3s}{2}} \leq|y|<\infty$)}
	\subsubsection{Closing of $\eqref{eq:4.2}_{2}$}
	Let  $V:=e^{s} \partial_{y} W $, we will show
	\begin{equation*}
	|V\circ \Phi_{W}^{y_{0}}(s)| \leq \frac{3}{2},
	\end{equation*}
	which verifies $\eqref{eq:4.1}_{2}$. 
	From $\eqref{eq:2.17}_{1}$, one may deduce that $ V $ satisfies
	\begin{equation*}\label{eq:9.23}
	(\partial_{s}+\beta_{\tau} \partial_{y} W) V+\mathcal{V}_{W} \partial_{y} V=e^{s}F_{W}^{(1)}.
	\end{equation*}
	For the damping,  by  $\eqref{eq:4.2}_{3} $, we bound it as 
	\begin{equation*}
	\mathcal{D}=\beta_{\tau} \partial_{y} W \geq-(1+4Me^{-s}) \times 2 e^{-s} \geq-3 e^{-s},
	\end{equation*}
	which gives
	\begin{equation*}\label{eq:9.24}
	e^{-\int_{s_{*}}^{s} \mathcal{D} \circ \Phi_{W}^{y_{0}}(s^{\prime}) d s^{\prime}} \leq e^{3 e^{-s *}} \leq e^{3 \delta} \leq \frac{5}{4} .
	\end{equation*}
	For the forcing term, one applies \eqref{eq:4.2} and \eqref{eq:4.20} to estimate 
	\begin{equation*}\label{eq:9.25}
	\begin{aligned}
	e^{s}|F_{W}^{(1)}|&\lesssim  e^{\frac{s}{2}}|\partial_{y}B|+e^{s}|\partial_{y}G_{W}||\partial_{y}W|\\
    &\lesssim e^{-s}|\partial^{2}_{x}b|+2Me^{-\frac{s}{3}}\leq e^{-\frac{s}{6}}.
	\end{aligned}
	\end{equation*}
    Here we also use the chain rule and the regularity of $b$ to get an additional decay.
    
	Putting the above estimates together, we get
	\begin{equation*}
	\begin{aligned}
	|V \circ \Phi_{W}^{y_{0}}(s)| & \leq \frac{5}{4}|V \circ\Phi_{W}^{y_{0}}(s_{*})|+C \int_{s_{*}}^{s} e^{-\frac{s^\prime}{6}} d s^{\prime}\\
    &\leq \frac{5}{4}|V \circ \Phi_{W}^{y_{0}}(s_{*})|+C \delta^{\frac{1}{6}}.
	\end{aligned}
	\end{equation*}
	(I) If $\frac{1}{2} e^{\frac{3 s}{2}} \leq|y_{0}|<\infty$  and $ s_{*}=s_{0} $, then one uses  $\eqref{eq:3.3}_{3}$ to find
	\begin{equation*}
	|V \circ \Phi_{W}^{y_{0}}(s)| \leq \frac{5}{4} e^{s_{0}} \delta+C \delta^{\frac{1}{6}} \leq \frac{3}{2}.
	\end{equation*}
	(II) If $ s_{*}>s_{0}$  and $ |y_{0}|=\frac{1}{2} e^{\frac{3 s_{*}}{2}} $, one instead uses  $\eqref{eq:4.2}_{2}$ to deduce
	\begin{equation*}
	\begin{aligned}
	|V\circ \Phi_{W}^{y_{0}}(s)| & \leq \frac{5}{4} e^{s_{*}}\bigg(\delta^{\frac{1}{18}}+\frac{7}{20}\bigg)\langle y_0 \rangle^{-\frac{2}{3}}+C \delta^{\frac{1}{6}}  \\
	& \leq \frac{5}{4} \times \frac{2}{5} \times 2^{\frac{2}{3}} e^{s_{*}} e^{-s_{*}}+C \delta^{\frac{1}{6}} \leq \frac{3}{2},
	\end{aligned}
	\end{equation*}
	where one has used \eqref{eq:2.16} to get a sharp estimate
	\begin{equation*}
	\begin{aligned}
	|\partial_{y} W(y, s)| & \leq|\partial_{y} \widetilde{W}(y, s)|+|\partial_{y} \widetilde{W}(y)| \leq \bigg(\delta^{\frac{1}{18}} +\frac{7}{20}\bigg)\langle y_0 \rangle^{-\frac{2}{3}}.
	\end{aligned}
	\end{equation*}

	\subsection{Closing of $\eqref{eq:4.3}_{2}$} 
	The bootstrap assumption $\eqref{eq:4.3}_{2}$ can be verified by showing
	\begin{equation*}
	|\partial_{y}^{2} W\circ \Phi_{W}^{y_{0}}(s)| \leq \frac{3}{4} M^{\frac{1}{5}}.
	\end{equation*}
	For such purpose, we consider the equation $\eqref{eq:2.17}_{1}$ with $n=2$ which reads 
	\begin{equation*}
	\bigg(\partial_{s}+\underbrace{\frac{5}{2} +3 \beta_{\tau} \partial_{y} W}_{\mathcal{D}}\bigg) \partial_{y}^{2}W+\mathcal{V}_{W} \cdot\partial_{y}^{3} W=F_{W}^{(2)}.
	\end{equation*}
	The forcing term can be estimated by
	\begin{equation*}
	\begin{aligned}
	|F_{W}^{(2)}|&\lesssim |e^{-\frac{s}{2}}\partial^{2}_{y}B|+|\partial^{2}_{y} G_{W}||\partial_{y} W|+|\partial_{y} G_{W}||\partial^{2}_{y} W|\\
	&\lesssim Me^{-\frac{7}{2}s}+Me^{-\frac{23}{18}s}+Me^{-\frac{5}{18}s}+M^{\frac{6}{5}}e^{-\frac{s}{3}}\leq e^{-\frac{s}{6}}.
	\end{aligned}
	\end{equation*}
	To handle the damping, there are two cases to consider depending on the size of $|y|$.
	
	\underline{Case 1: $ \ell \leq|y| \leq \frac{1}{2} e^{\frac{3 s}{2}} $.} By \eqref{eq:2.14} and $ \eqref{eq:4.2}_{2} $, the damping is bounded as 
	\begin{equation*}
	\mathcal{D} \geq \frac{5}{2}-3(1+4Me^{-\frac{s}{3}})(1+\delta^{\frac{1}{18}})\langle y \rangle^{-\frac{2}{3}} \geq -3\langle y \rangle^{-\frac{2}{3}}.
	\end{equation*}
	Then we have
	\begin{equation*}
	\begin{aligned}
	|\partial_{y}^{2} W \circ \Phi_{W}^{y_{0}}(s)| & \leq \ell^{-30}|\partial_{y}^{2} W \circ \Phi_{W}^{y_{0}}(s_{*})|+\ell^{-30} \int_{s_{*}}^{s} e^{-\frac{s^{\prime}}{6}} d s^{\prime} \\
	& \leq \ell^{-30}(|\partial_{y}^{2} W \circ \Phi_{W}^{y_{0}}(s_{*})|+2 \delta^{\frac{1}{6}}),
	\end{aligned}
	\end{equation*}
	in which the last term can be further estimated as follows:\\
	(I)  $\ell \leq|y_{0}| \leq \frac{1}{2} e^{\frac{3 s}{2}}$  and  $s_{*}=s_{0} $. It follows from   $\eqref{eq:3.4}_{2}$   that
	\begin{equation*}
	|\partial_{y}^{2} W \circ \Phi_{W}^{y_{0}}(s)| \leq \ell^{-30}(M^{\frac{1}{10}}+2 \delta^{\frac{1}{6}}) \leq \frac{3}{4} M^{\frac{1}{5}}.
	\end{equation*}
	(II) $ s_{*}>s_{0}$  and $ |y_{0}|=\ell $. By  $\eqref{eq:4.3}_{1}$, one has
	\begin{equation*}
	|\partial_{y}^{2} W \circ \Phi_{W}^{y_{0}}(s)| \leq \ell^{-30}\bigg(\frac{1}{2} \delta^{\frac{1}{12}} \ell^{2}+\langle \ell \rangle^{-\frac{5}{3}}+2 \delta^{\frac{1}{6}}\bigg) \leq \frac{3}{4} M^{\frac{1}{5}} .
	\end{equation*}
	
	\underline{Case 2:  $\frac{1}{2} e^{\frac{3 s}{2}} \leq|y|<\infty$.} One uses  $\eqref{eq:4.2}_{3}$  to bound the damping below by
	\begin{equation*}
	\mathcal{D} \geq \frac{5}{2}-3 \times \frac{101}{100} \times 2 e^{-s} \geq 2.
	\end{equation*}
	Hence we get
	\begin{equation*}
	\begin{aligned}
	|\partial_{y}^{2} W \circ \Phi_{W}^{y_{0}}(s)| & \leq e^{-2(s-s_{*})}\bigg(|\partial_{y}^{2} W \circ \Phi_{W}^{y_{0}}(s_{*})| +\int_{s_{*}}^{s} e^{-\frac{s^{\prime}}{6}} d s^{\prime}\bigg) \\
	& \leq e^{-2(s-s_{*})}\big(|\partial_{y}^{2} W \circ \Phi_{W}^{y_{0}}(s_{*})| +2 \delta^{\frac{1}{6}}\big)\\
	&\leq M^{\frac{1}{10}}+2\delta^{\frac{1}{6}}\leq \frac{3}{4} M^{\frac{1}{5}}.
	\end{aligned}
	\end{equation*}

	\section{Proof of Theorem \ref{thm1}}\label{sec10}
	
	Based on the preliminaries in the previous sections, we can achieve the proof of Theorem \ref{thm1}. The proof in this section is close to \cite{MR4321245,MR4612576}, we include it here for the sake of completeness.
	\subsection{Precise blowup information of the solution}
	\
	\par
	
	1. \underline{Blowup time and location.} The blowup time and location have been obtained in \eqref{eq:6.2} and \eqref{eq:6.5}, respectively.
	
	2.  \underline{$L^{\infty}$  bound of $ w $.} It follows from \eqref{eq:6.1} that
	\begin{equation*}
	\|w(\cdot, t)\|_{L^{\infty}}=\|e^{-\frac{s}{2}} W+\kappa\|_{L^{\infty}} \leq M \quad \text { for } t \in[-\delta, T_{*}] .
	\end{equation*}
	
	3. \underline{Blowup rate of  $\partial_{x} w $.} We first claim that
	\begin{equation}\label{eq:10.1}
	\frac{1}{2}(T_{*}-t) \leq \tau(t)-t \leq 2(T_{*}-t).
	\end{equation}
	The first and second inequality is equivalent to
	\begin{equation}\label{eq:10.2}
	T_{*} \leq 2 \tau(t)-t
	\end{equation}
	and
	\begin{equation}\label{eq:10.3}
	\tau(t)+t \leq 2 T_{*},
	\end{equation}
	respectively. Both of \eqref{eq:10.2} and \eqref{eq:10.3} can be verified by the fact that $2 \tau(t)-t$  and $ \tau(t)+t$ is monotone decreasing and increasing due to \eqref{eq:4.8}, together with $\tau(T_{*})=T_{*} $, respectively.
	
	On the one hand, by \eqref{eq:2.22}, one calculates that
	\begin{equation}\label{eq:10.4}
	\partial_{x} w(\xi(x), t)=\frac{1}{\tau(t)-t} \partial_{y} W(0, s)=-\frac{1}{\tau(t)-t}.
	\end{equation}
	It follows from \eqref{eq:10.1} and \eqref{eq:10.4} that
	\begin{equation*}
	\frac{1}{2(T_{*}-t)} \leq|\partial_{x} w(\xi(x), t)| \leq \frac{2}{T_{*}-t},
	\end{equation*}
	which means that
	\begin{equation*}
	\lim _{t \rightarrow T_{*}} \partial_{x} w(\xi(x), t)=-\infty .
	\end{equation*}
	So  $\partial_{x} w $ blows up at  $x_{*}=\xi(T_{*}) $.
	On the other hand, observing
	\begin{equation}\label{eq:10.5}
	\partial_{x} w(x, t)=\frac{1}{\tau(t)-t} \partial_{y} W(y, s),
	\end{equation}
	one deduces from \eqref{eq:4.12} and \eqref{eq:10.1} that
	\begin{equation*}
	\frac{1}{2(T_{*}-t)} \leq\|\partial_{x} w(\cdot, t)\|_{L^{\infty}} \leq \frac{2}{T_{*}-t} .
	\end{equation*}

	4. \underline{The regularity of  $w$  at  $T_{*} $.} We first claim that if  $|x-x_{*}|>\frac{1}{2} $, then
	\begin{equation}\label{eq:10.6}
	|\partial_{x} w(x, T_{*})| \leq 2 .
	\end{equation}
	To verify \eqref{eq:10.6}, one first sees that there exists  $t_{1} \in[-\epsilon, T_{*})$  such that
	\begin{equation*}
	|x-\xi(t)| \geq \frac{1}{2} \quad \text { for } t \in[t_{1}, T_{*}],
	\end{equation*}
	which implies in terms of the self-similar variables that
	\begin{equation*}
	|y| \geq \frac{1}{2(\tau(t)-t)^{\frac{3}{2}}}=\frac{1}{2} e^{\frac{3}{2} s} \quad \text { for } t \in[t_{1}, T_{*}).
	\end{equation*}
	Then this together with  $\eqref{eq:4.2}_{3}$ leads to
	\begin{equation*}
	|\partial_{x} w(x, t)|=e^{s}|\partial_{y} W(y, s)| \leq 2 \quad \text { for } t \in[t_{1}, T_{*}),
	\end{equation*}
	which proves \eqref{eq:10.6} by sending  $t$ to  $T_{*}$  in the above inequality.
	
	We next claim that if $ 0<|x-x_{*}| \leq \frac{1}{2} $, then
	\begin{equation}\label{eq:10.7}
	|\partial_{x} w(x, T_{*})| \sim|x-x_{*}|^{-\frac{2}{3}} .
	\end{equation}
	Indeed, in this case, there exists  $t_{2} $ close to  $T_{*}$  in  $[-\epsilon, T_{*})$  such that
	\begin{equation*}
	\frac{1}{2}|x-x_{*}| \leq|x-\xi(t)| \leq|x-x_{*}| \quad \text { for } t \in[t_{2}, T_{*}],
	\end{equation*}
	which implies
	\begin{equation}\label{eq:10.8}
	\frac{1}{2}|x-x_{*}| e^{\frac{3}{2} s} \leq|y| \leq|x-x_{*}| e^{\frac{3}{2} s} \leq \frac{1}{2} e^{\frac{3}{2} s} \quad \text { for } t \in[t_{1}, T_{*}) .
	\end{equation}
	By choosing a larger  $t_{2}$  if necessary, we may also assume $ \ell \leq|y| $. It follows from  $\eqref{eq:4.2}_{2}$, 
	\eqref{eq:10.5}  and \eqref{eq:10.8} that
	\begin{equation*}
	|\partial_{x} w(x, t)| \lesssim \frac{1}{\tau(t)-t}|y|^{-\frac{2}{3}} \lesssim|x-x_{*}|^{-\frac{2}{3}},
	\end{equation*}
	and
	\begin{equation*}
	|\partial_{x} w(x, t)| \gtrsim \frac{1}{\tau(t)-t}|y|^{-\frac{2}{3}} \gtrsim|x-x_{*}|^{-\frac{2}{3}}
	\end{equation*}
	for all  $t \in[t_{1}, T_{*}) $.
	Then \eqref{eq:10.7} follows by sending $t$  to  $T_{*}$ in the above two inequalities.
	
	Finally, we conclude from \eqref{eq:10.6} and \eqref{eq:10.7} that  $w(\cdot, T_{*}) \in C^{\frac{1}{3}}(\mathbb{R})$  and  $w$  has a cusp singularity at $ (x_{*}, T_{*}) $.
	
	\subsection{Asymptotic Convergence to Stationary Solution}
	\
	\par
	\underline{Step 1.} It is straightforward to verify, as shown in \eqref{eq:9.9}, that the limit $\nu = \lim_{s \rightarrow \infty} \partial_{y}^{3} W(0, s)$ exists. Let
	$\widetilde{W}_{\nu}:=W-\overline{W}_{\nu}$.
	To show \eqref{eq:3.9}, it suffices to verify
	\begin{equation}\label{eq:10.9}
	\limsup _{s \rightarrow \infty}|\widetilde{W}_{\nu}(y, s)|=0 \quad \text { for } y \in \mathbb{R}.
	\end{equation}
	Recalling \eqref{eq:3.1} and \eqref{eq:3.10}, one sees that \eqref{eq:10.9} is obvious when $y=0$. Thus, it remains to prove \eqref{eq:10.9} for $ y \neq 0 $.
	
	We begin by proving \eqref{eq:10.9} in the near field, specifically for $|y|\in(0,|y_{0}|)$ for some small $|y_{0}|$. Recalling the constrains \eqref{eq:2.22}, we can express $\widetilde{W}_{\nu}$ using Taylor expansion as follows:
	\begin{equation}\label{eq:10.10}
	\widetilde{W}_{\nu}(y, s)=\frac{y^{3}}{6} \partial_{y}^{3} \widetilde{W}_{\nu}(0, s)+\frac{y^{4}}{24} \partial_{y}^{4} \widetilde{W}_{\nu}(Y, s)
	\end{equation}
	for some  $Y$ between $0$ and  $y$  when  $y$ is close to zero. Next, we handle the RHS of \eqref{eq:10.10} term by term. Observe that
	$$\lim _{s \rightarrow \infty} \partial_{y}^{3} \widetilde{W}_{\nu}(0, s)=\lim _{s \rightarrow \infty} \partial_{y}^{3} W(0, s)-\nu=0.$$
	Hence, first fixing a small positive constant $|y_{0}|>0 $, and then choosing a  small  $\delta^{\prime} \in(0, M|y_{0}|) $, there exists a large enough  $s_{0}=s_{0}(y_{0}, \delta^{\prime})$  such that the following holds:
	\begin{equation}\label{eq:10.11}
	|\partial_{y}^{3} \widetilde{W}_{\nu}(y_{0}, s)| \leq \delta^{\prime} \quad\text{for}\ s \geq s_{0}.
	\end{equation}
	Notice that
	\begin{equation*}
	\partial_{y}^{4} \overline{W}_{\nu}(y)=\Big(\frac{\nu}{6}\Big)^{\frac{3}{2}} \partial_{y}^{4} \overline{W}\bigg(\Big(\frac{\nu}{6}\Big)^{\frac{1}{2}} y\bigg).
	\end{equation*}
	Then we have
	\begin{equation}\label{eq:10.12}
	\begin{aligned}
	\|\partial_{y}^{4} \widetilde{W}_{\nu}\|_{L^{\infty}} & \leq\|\partial_{y}^{4} W\|_{L^{\infty}}+\|\partial_{y}^{4} \overline{W}_{\nu}\|_{L^{\infty}} \\
	& \leq M+100\Big(\frac{\nu}{6}\Big)^{\frac{3}{2}} \leq 2 M.
	\end{aligned}
	\end{equation}
	It follows from \eqref{eq:10.10}-\eqref{eq:10.12} that
	\begin{equation}\label{eq:10.13}
	|\widetilde{W}_{\nu}(y_{0}, s)| \leq \frac{\delta^{\prime}|y_{0}|^{3}}{6}+\frac{M|y_{0}|^{4}}{12} \quad \text {for}\ s \geq s_{0}.
	\end{equation}

	\underline{Step 2.} Basic calculations indicate that $\widetilde{W}_{\nu}$ evolves as follows:
	\begin{equation}\label{eq:10.14}
	\bigg(\partial_{s}-\frac{1}{2}+\partial_{y} \overline{W}_{\nu}\bigg) \widetilde{W}_{\nu}+\bigg(\underbrace{\frac{3}{2} y+W}_{P}\bigg) \partial_{y} \widetilde{W}_{\nu}=F_{\widetilde{W}_{\nu}}
	\end{equation}
	with
	$$F_{\widetilde{W}_{\nu}}=\underbrace{-\frac{3}{4}\beta^{*} \beta_{\tau}e^{-\frac{s}{2}}B-\beta_{\tau}e^{-\frac{s}{2}} \dot{\kappa}}_{I_1}\underbrace{-\big(G_{W}+\beta_{\tau}\dot{\tau}W \partial_{y} W\big)\partial_{y} W}_{I_{2}}$$
	
	We first handle the forcing terms. For $I_1$, it can be shown that
	\begin{equation*}
	|I_{1}| \lesssim Me^{-\frac{s}{2}}+Me^{-\frac{5}{18}s}\leq e^{-\frac{2}{9}s}.
	\end{equation*}
	To estimate $I_2$, one shall consider three cases depending on the scales of $|y|$.\\
	(I) When $0\leq |y|\leq \ell$, we have 
	\begin{equation*}
	|I_2|\lesssim  M\Big(e^{-\frac{5}{18} s}+|\ell|e^{-\frac{s}{3}}+(\delta^{\frac{1}{12}}\ell^{4}+1)e^{-\frac{s}{3}}\Big) (\delta^{\frac{1}{12}}\ell^3+1)\leq e^{-\frac{2}{9}s}.
	\end{equation*}
	(II) When $\ell\leq |y|\leq \frac{1}{2}e^{\frac{3}{2}s}$, it follows from $\eqref{eq:2.15}$, $\eqref{eq:4.2}_2$ and \eqref{eq:4.21} that 
	\begin{equation*}
	|I_2|\lesssim M\Big(e^{-\frac{5}{18}s}+|\ell|e^{-\frac{s}{3}}+(1+\delta^{\frac{1}{15}})\langle y\rangle^{\frac{1}{3}} e^{-\frac{s}{3}}\Big)(1+\delta^{\frac{1}{18}})\langle y\rangle^{-\frac{2}{3}}\leq e^{-\frac{2}{9}s}.
	\end{equation*}
	(III) When $|y|\geq \frac{1}{2}e^{\frac{3}{2}s}$, we find
	\begin{equation*}
	|I_2|\lesssim(5M e^{\frac{s}{2}}+2Me^{\frac{s}{2}}\cdot Me^{-\frac{s}{3}})\cdot2e^{-s}\leq e^{-\frac{2}{9}s},
	\end{equation*}
	where $\eqref{eq:4.2}_3$ and \eqref{eq:4.22} have been used.
	Combining all estimates yields
	\begin{equation}\label{eq:10.15}
	|F_{\widetilde{W}_{\nu}}|\leq 2 e^{-\frac{2}{9}s}.
	\end{equation}
	
	To estimate $\widetilde{W}_{\nu} $, we need to work with a new Lagrangian trajectory associated with $P $. To this end, we consider the following Lagrangian trajectory:
	\begin{equation}\label{eq:10.16}
	\frac{d}{d s} \Psi^{y_{0}}(s)=\frac{3}{2} \Psi^{y_{0}}(s)+P \circ \Psi^{y_{0}}(s), \quad \Psi^{y_{0}}(s_{0})=y_{0}
	\end{equation}
	By applying the mean value theorem and using \eqref{eq:2.22}, we obtain
	\begin{equation}\label{eq:10.17}
	|W(y, s)| \leq|W(0, s)|+\|\partial_{y} W\|_{L^{\infty}}|y| \leq \frac{101}{100}|y|.
	\end{equation}
	It follows from \eqref{eq:10.16} and \eqref{eq:10.17} that
	\begin{equation*}
	\frac{d}{d s}|\Psi^{y_{0}}(s)|^{2} \geq 3|\Psi^{y_{0}}(s)|^{2}-2 \times \frac{101}{100}|\Psi^{y_{0}}(s)|^{2} \geq \frac{4}{5}|\Psi^{y_{0}}(s)|^{2},
	\end{equation*}
	which, together with  $\Psi^{y_{0}}(s_{0})=y_{0}$, gives
	\begin{equation}\label{eq:10.18}
	|\Psi^{y_{0}}(s)| \geq|y_{0}| e^{\frac{2}{5}(s-s_{0})}.
	\end{equation}

	Let $G(y, s)=e^{-\frac{3}{2}(s-s_{0})} \widetilde{W}_{\nu}(y, s) $. According to \eqref{eq:10.14}, it is straightforward to calculate that
	\begin{equation}\label{eq:10.19}
	\bigg(\frac{d}{d s}+1+\partial_{y} \overline{W}_{\nu}\bigg) G \circ \Psi^{y_{0}}(s)= e^{-\frac{3}{2}(s-s_{0})} F_{\widetilde{W}_{\nu}} \circ \Psi^{y_{0}}(s),
	\end{equation}
	with the damping bounded from below by 0 due to \eqref{eq:one-order}:
	$$1+\partial_{y} \overline{W}_{\nu} \geq 1-\|\partial_{y} \overline{W}_{\nu}\|_{L^{\infty}}=1-\|\partial_{y} \overline{W}\|_{L^{\infty}}=0.$$

	Plugging \eqref{eq:10.13} and \eqref{eq:10.15} into \eqref{eq:10.19}, one obtains
	\begin{equation*}
	\begin{aligned}
	|G \circ \Psi^{y_{0}}(s)| & \leq|G \circ \Psi^{y_{0}}(s_{0})|+\beta_{\tau} \int_{s_{0}}^{s}|F_{\widetilde{W}_{\nu}} \circ \Psi^{y_{0}}(s_{0})| e^{-\frac{3}{2}(s^{\prime}-s_{0})} d s^{\prime} \\
	& \leq \frac{\delta^{\prime}|y_{0}|^{3}}{6}+\frac{M|y_{0}|^{4}}{12}+\int_{s_{0}}^{s} e^{-\frac{2s^{\prime}}{9}} e^{-\frac{3}{2}(s^{\prime}-s_{0})} d s^{\prime} \\
	& \leq \frac{\delta^{\prime}|y_{0}|^{3}}{6}+\frac{M|y_{0}|^{4}}{12}+\frac{M|y_{0}|^{4}}{12} \leq M|y_{0}|^{4},
	\end{aligned}
	\end{equation*}
	which implies
	\begin{equation}\label{eq:10.20}
	|\widetilde{W}_{\nu} \circ \Psi^{y_{0}}(s)| \leq M|y_{0}|^{4} e^{-\frac{3}{2}(s-s_{0})}.
	\end{equation}

	Let $s_{*}=s_{0}+\frac{13}{5} \log |y_{0}|^{-1} $. It follows from \eqref{eq:10.20} that
	\begin{equation}\label{eq:10.21}
	|\widetilde{W}_{\nu} \circ \Psi^{y_{0}}(s)| \leq M|y_{0}|^{4-\frac{3}{2}\cdot \frac{13}{5}} \leq M|y_{0}|^{\frac{1}{10}}
	\end{equation}
	for $ s \in(s_{0}, s_{*}) $.
	For all  $y$  between  $y_{0}$  and  $\Psi^{y_{0}}(s_{*}) $, there exists $ s \in(s_{0}, s_{*})$ such that $y=\Psi^{y_{0}}(s) $. Therefore, for such $ (y, s)$, \eqref{eq:10.21}  gives
	\begin{equation}\label{eq:10.22}
	|\widetilde{W}_{\nu}(y, s)| \leq M|y_{0}|^{\frac{1}{10}}.
	\end{equation}
	By \eqref{eq:10.18}, one can infer that \eqref{eq:10.22} will cover at least all $ y$  satisfying
	\begin{equation}\label{eq:10.23}
	|y_{0}| \leq|y| \leq|y_{0}| e^{\frac{2}{5} \cdot \frac{13}{5} \log |y_{0}|^{-1}}=|y_{0}|^{-\frac{1}{25}}.
	\end{equation}
	Hence, with \eqref{eq:10.22} and \eqref{eq:10.23} established, one can take the limit as $s_{0} \rightarrow \infty $ to obtain
	\begin{equation}\label{eq:10.24}
	\underset{s \rightarrow \infty}{\limsup }|\widetilde{W}_{\nu}(y, s)| \leq M|y_{0}|^{\frac{1}{10}}.
	\end{equation}
	Finally, sending $y_{0} \rightarrow 0$ in \eqref{eq:10.24} proves for all  $y \neq 0$  that
	\begin{equation*}
	\limsup _{s \rightarrow \infty}|\widetilde{W}_{\nu}(y, s)|=0.
	\end{equation*}
	This completes the proof of \eqref{eq:10.9}.

	\section*{Acknowledgments}
	Y. Wang was supported by the Grant No. 830018 from China. J.-C. Saut was partially supported by the ANR project ISAAC (ANPG2023).

	

\end{document}